\documentclass{amsart}
\title{Realizations of multiassociahedra via bipartite rigidity}
\author{Luis Crespo Ruiz}

\usepackage{color,graphicx}

\usepackage[colorinlistoftodos]{todonotes}

\usepackage{url}
\usepackage[colorlinks=true]{hyperref}

\newtheorem{theorem}{Theorem}[section]
\newtheorem{lemma}[theorem]{Lemma}
\newtheorem{proposition}[theorem]{Proposition}
\newtheorem{corollary}[theorem]{Corollary}
\newtheorem{conjecture}[theorem]{Conjecture}
\theoremstyle{definition}
\newtheorem{definition}[theorem]{Definition}
\newtheorem{remark}[theorem]{Remark}

\newcommand{\ocultar}[1]{}

\newcommand{\N}{\mathbb{N}}
\newcommand{\R}{\mathbb{R}}
\newcommand{\proj}{\mathbb{P}}
\newcommand{\K}{\mathbb{K}}

\newcommand{\C}{\mathbb{C}}
\newcommand{\bnn}{{\binom{[n]}2}}
\newcommand{\p}{\mathbf{p}}

\newcommand{\HH}{\mathcal{H}}
\newcommand{\PP}{\mathcal{P}}
\newcommand{\VV}{\mathcal{V}}

\newcommand{\bt}{\mathbf{t}}
\newcommand{\M}{\mathcal{M}}

\newcommand{\Gr}[2]{\mathcal{G}\hskip-1pt\mathit{r}_#1(#2)}
\newcommand{\Pf}[2]{\mathcal{P}\hskip-1pt\mathit{f}\hskip-1pt_{#1}(#2)}
\newcommand{\Ass}[2]{\mathcal{A}\hskip-0pt\mathit{ss}_{#1}(#2)}
\newcommand{\OvAss}[2]{\overline{\mathcal{A}\hskip-0pt\mathit{ss}}_{#1}(#2)}
\newcommand{\PV}[2]{\mathrm{Pf}_{#1}(#2)}
\newcommand{\PVplus}[2]{\mathrm{Pf}_{#1}^+(#2)}
\newcommand{\g}{\textbf{g}}
\newcommand{\Grob}{\operatorname{Grob}}
\newcommand{\link}{\operatorname{lk}}

\newcommand{\Mplus}[2]{M_{#1}^+(#2)}

\DeclareMathOperator{\trop}{trop}
\DeclareMathOperator{\ini}{in}
\DeclareMathOperator{\Supp}{Supp}
\DeclareMathOperator{\sign}{sign}

\begin{document}
	\begin{abstract}
		Let $\Ass{k}{n}$ denote the simplicial complex of $(k+1)$-crossing-free subsets of edges in $\bnn$. Here $k,n\in \N$ and $n\ge 2k+1$. It is conjectured that this simplicial complex is polytopal (Jonsson 2005). However, despite several recent advances, this is still an open problem.
		
		In this paper we attack this problem using as a vector configuration the rows of a rigidity matrix, namely, hyperconnectivity restricted to bipartite graphs. We see that in this way $\Ass{k}{n}$ can be realized as a polytope for $k=2$ and $n\le 10$, and as a fan for $k=2$ and $n\le 13$, and for $k=3$ and $n\le 11$. However, we also prove that the cases with $k\ge 3$ and $n\ge \max\{12,2k+4\}$ are not realizable in this way.
		
		We also give an algebraic interpretation of the rigidity matroid,  relating it to a projection of determinantal varieties with implications in matrix completion, and prove the presence of a fan isomorphic to $\Ass{k-1}{n-2}$ in the tropicalization of that variety.
	\end{abstract}
	
\maketitle
\tableofcontents
\section{Introduction}
Consider the simplicial complex $\Ass{1}{n}$ whose vertices are the diagonals of the $n$-gon, and whose faces are the non-crossing sets of diagonals. The facets of this simplicial complex are the \textit{triangulations} of the $n$-gon. Ignoring the sides of the polygon $\{i,i+1\}$, that are common to all triangulations, we obtain a polytopal sphere of dimension $n-4$, whose dual is commonly called the \textit{associahedron}.

Now we will consider the general case where we allow up to $k$ diagonals to pairwise cross but not $k+1$ (assuming $n>2k+1$). A subset of $\bnn$ is called $(k+1)$-free if it does not contain $k+1$ edges that pairwise cross, and a $k$-triangulation if it is a maximal $(k+1)$-free set. We can now define a simplicial complex $\Ass{k}{n}$ whose faces are $(k+1)$-free sets (and whose facets are $k$-triangulations).

Diagonals of length at most $k$ (where length is measured cyclically) cannot participate in any $(k+1)$-crossing. Hence, it makes sense to delete them from $\Ass{k}{n}$ to obtain the reduced complex $\OvAss{k}{n}$, which is called  \emph{multiassociahedron} or {$k$-associahedron}. 
See Section~\ref{sec:multi} for more precise definitions, and \cite{PilPoc,PilSan,Stump} for additional information.

It was proved in \cite{Naka,DKM} that every $k$-triangulation of the $n$-gon has exactly $k(2n-2k-1)$ diagonals. That is, $\Ass{k}{n}$ is pure of dimension $k(2n-2k-1)-1$.
Jonsson~\cite{Jonsson} further proved that the reduced version $\OvAss{k}{n}$ is a shellable sphere of dimension $k(n-2k-1)-1$, and  conjectured it to be the normal fan of a polytope.

\begin{conjecture}[\cite{Jonsson}]
	\label{conj:polytope}
	For every $n\ge 2k+1$ the complex $\OvAss{k}{n}$ is a polytopal sphere. That is, there is a simplicial polytope of dimension $k(n-2k-1)-1$ and with $\binom{n}2-kn$ vertices whose lattice of proper faces is isomorphic to $\OvAss{k}{n}$.
\end{conjecture}

Conjecture~\ref{conj:polytope} is easy to prove  for $n\le 2k+3$. $\OvAss{k}{2k+1}$ is indeed a $-1$-sphere (the complex whose only face is the empty set). $\OvAss{k}{2k+2}$ is the face poset of a $(k-1)$-simplex, and $\OvAss{k}{2k+3}$ is (the polar of) the cyclic polytope of dimension $2k-1$ with $n$ vertices (Lemma 8.7 in \cite{PilSan}). 
The additional cases for which Jonsson's conjecture is known to hold are $k=2$ and $n=8$ \cite{bokpil}, $k=2$ and $n=9,10$ \cite{CreSan:Multi} and $k=3$ and $n=10$ \cite{CreSan:Multi}. In some additional cases $\OvAss{k}{n}$ has been realized as a complete simplicial fan, but it is open whether this fan is polytopal. This includes the cases $n\le 2k+4$ \cite{bcl}, the cases $k=2$ and $n\le 13$ \cite{Manneville} and the cases $k=3$ and $n\le 11$ \cite{bcl}.

In \cite{CreSan:Multi}, $\Ass{k}{n}$ is realized as a fan for $k=2,3$ and various values of $n$ using as rays the row vectors of a rigidity matrix of $n$ points in dimension $2k$, which has exactly the required rank $k(2n-2k-1)$ for $\Ass{k}{n}$. 
There are several versions of rigidity that can be used, most notably classical (bar-and-joint) rigidity, Kalai's \emph{hyperconnectivity}~\cite{Kalai}, and Billera-Whiteley's \emph{cofactor rigidity}~\cite{Whiteley}. Those new cases of multiassociahedra $\OvAss{k}{n}$ are realized with cofactor rigidity taking points along the parabola (which is known to be equivalent to bar-and-joint rigidity with points along the moment curve): the complex can be realized as a polytope for $n\le 10$ and as a complete fan for $n\le 13$, except $\OvAss{3}{12}$ and $\OvAss{3}{13}$.
On the other hand it is shown that certain multiassociahedra, namely those with $k\ge 3$ and $n\ge 2k+6$, cannot be realized as fans with cofactor rigidity, no matter how we choose the points. 

Interest in  polytopality of $\OvAss{k}{n}$ comes also from cluster algebras and Coxeter combinatorics. 
Let $w\in W$ be an element in a Coxeter group $W$ and let $Q$ a word of a certain length $N$. Assume that $Q$ contains as a subword a reduced expression for $w$. The \emph{subword complex} of $Q$ and $w$ is the simplicial complex with vertex set $[N]$ and with faces the subsets of positions that can be deleted from $Q$ and still contain a reduced expression for $w$. Knutson and Miller~\cite[Theorem 3.7 and Question 6.4]{KnuMil} proved that every subword complex is either a shellable ball or sphere, and they asked whether all spherical subword complexes are polytopal. 
It was later proved by Stump~\cite[Theorem 2.1]{Stump} that  $\OvAss{k}{n}$ is a spherical subword complex for the Coxeter system of type $A_{n-2k-1}$ and, moreover, it is \emph{universal}: every other spherical subword complex of type $A$ appears as a link in some $\OvAss{k}{n}$~\cite[Proposition 5.6]{PilSan:brick}. In particular, Conjecture~\ref{conj:polytope} is equivalent to a positive answer in type A to the question of Knutson and Miller. (Versions of $k$-associahedra for the rest of finite Coxeter groups exist, with the same implications~\cite{CLS14}).

In the case $k=1$, one way of realizing the associahedron is as the positive part of the space of ``tree metrics'', which coincides with the tropicalization $\trop(\Gr2n)$ of the Grassmannian $\Gr2n$ (see \cite{SpeStu,SpeWil}). More precisely:

\begin{theorem}[\protect{\cite[Section 5]{SpeWil}}]
	The totally positive tropical Grassmannian $\trop^+(\Gr2n)$ is a simplicial fan isomorphic to (a cone over) the extended associahedron $\OvAss{1}{n}$.
\end{theorem}

Let us briefly recall what the tropicalization of a variety, and its positive part, are. (See also \cite{BLS}). Let $I\subset\K[x_1,\dots,x_N]$ be a polynomial ideal and let $V=V(I)\subset \K^N$ be its corresponding variety. Each vector $v\in \R^N$, considered as giving weights to the variables, defines an initial ideal $\ini_v(I)$, consisting of the initial forms $\ini_v(f)$ of the polynomials in $f$. For the purposes of this paper we take the following definitions. (These are not the standard definitions, but are equivalent to them as shown for example in \cite[Propositions 2.1 and 2.2]{SpeWil}):

\begin{definition}
	\label{defi:tropical}
	The \emph{tropical variety} $\trop(V)$ of $V$ equals the set of $v\in \R^N$ for which $\ini_v(I)$ does not contain any monomial. If $\K=\C$, the \emph{totally positive part} of $\trop(V)$, denoted $\trop^+(V)$,  equals the set of $v\in \R^N$ for which $\ini_v(I)$ does not contain any polynomial with all coefficients real and positive.
\end{definition}

In another previous paper \cite{CreSan:Pfaffians} we prove that the relation between the associahedron and $\Gr2n$ extends to a relation between the multiassociahedron $\Ass{k}{n}$ and the tropical variety of Pfaffians of  degree $k+1$: the part of that variety contained in the Gr\"obner cone of $(k+1)$-crossings (that is strictly contained in the totally positive part) is a fan isomorphic to the $k$-associahedron in $n$ vertices. This opens another way for realizability of $\OvAss{k}{n}$ as a complete fan: finding a projection $\R^\bnn\to\R^{k(2n-2k-1)}$ that is injective in that fan, embedding it as a complete fan, so that the link of the ``irrelevant face'' is the fan that we want. We have achieved it for $k=1$, leading to the construction of the associahedron by $\g$-vectors (that was previously found in \cite{HPS}).

\subsection*{Summary of methods and results}
This paper tries to continue those ideas to realize the multiassociahedron, either with the rows of a rigidity matrix or as part of a tropical variety. However, in this case the graph is first converted into a bipartite graph by duplicating each vertex and distributing the edges at a vertex between the two copies of that vertex, depending on the edge going left or right. Also, to preserve symmetry, we will reverse the ordering of the right vertices (those connected to a vertex at their left).

In what follows, non-bipartite graphs will be denoted, as usual, as $G=([n],E)$, where $E\subset\bnn$, and bipartite graphs will be denoted as $G=([n_1]\cup[n_2]',E)$ where $[n_1]=\{1,\ldots,n_1\}$ are the \textit{left vertices}, $[n_2]'=\{1',\ldots,n_2'\}$ are the \textit{right vertices} and $E\subset[n_1]\times [n_2]$, so that $(i,j)$ denotes the edge between $i$ and $j'$.

\begin{definition}
	The \textit{bipartization} of a graph $G=([n],E)$ is the graph $G'=([n]\cup [n]',E')$ where $E'=\{(i,n+1-j):\{i,j\}\in E,i<j\}$.
\end{definition}

Then, our method to realize the multiassociahedron is to bipartize each $k$-triangulation and take as a vector configuration the rows of the hyperconnectivity matrix. The degree of the last $k+1$ vertices in each side is $k,k-1,k-2,\ldots,1,0$ and all those edges are irrelevant (or boundary) ones that will be in all $k$-triangulations, so we can ignore them in the bipartization:

\begin{definition}
	The \textit{reduced bipartization} of a $k$-triangulation is its bipartization restricted to $[n-k-1]\cup [n-k-1]'$.
\end{definition}

Figure~\ref{fig:bipartization}
shows a $2$-triangulation in $7$ vertices, its bipartization and its reduced bipartization. Note that this operation seems to introduce $3$-crossings, some of which remain when reducing. This is because, in the rightmost graph, the natural positions of the vertices $4$ and $4'$ are reversed: both come from the original vertex $4$. This hints at the fact that we should not consider things like $\{24',33',42'\}$ as $3$-crossings, as long as they are not so ($\{24,35,46\}$ in this case) in the original graph.

\begin{figure}[htb]
\tikzset{inner sep=0}
\begin{tikzpicture}
	\node[label=180:{1=7'}] (p1) at (-2,2) {$\bullet$};
	\node[label=210:{2=6'}] (p2) at (-1.7,1) {$\bullet$};
	\node[label=240:{3=5'}] (p3) at (-1,0.3) {$\bullet$};
	\node[label=270:{4=4'}] (p4) at (0,0) {$\bullet$};
	\node[label=300:{5=3'}] (p5) at (1,0.3) {$\bullet$};
	\node[label=330:{6=2'}] (p6) at (1.7,1) {$\bullet$};
	\node[label=0:{7=1'}] (p7) at (2,2) {$\bullet$};
	\draw (p1.center)--(p2.center)--(p3.center)--(p4.center)--(p5.center)--(p6.center)--(p7.center)--(p1.center);
	\draw (p1.center)--(p3.center)--(p5.center)--(p7.center)--(p2.center)--(p4.center)--(p6.center)--(p1.center);
	\draw (p1.center)--(p4.center)--(p7.center)--(p3.center)--(p6.center);
\end{tikzpicture}
\raisebox{1.4cm}{ $\rightarrow$ }
\begin{tikzpicture}
	\node[label=180:1] (p1) at (0,3) {$\bullet$};
	\node[label=180:2] (p2) at (0,2.5) {$\bullet$};
	\node[label=180:3] (p3) at (0,2) {$\bullet$};
	\node[label=180:4] (p4) at (0,1.5) {$\bullet$};
	\node[label=180:5] (p5) at (0,1) {$\bullet$};
	\node[label=180:6] (p6) at (0,0.5) {$\bullet$};
	\node[label=180:7] (p7) at (0,0) {$\bullet$};
	\node[label=0:1'] (p1') at (2,3) {$\bullet$};
	\node[label=0:2'] (p2') at (2,2.5) {$\bullet$};
	\node[label=0:3'] (p3') at (2,2) {$\bullet$};
	\node[label=0:4'] (p4') at (2,1.5) {$\bullet$};
	\node[label=0:5'] (p5') at (2,1) {$\bullet$};
	\node[label=0:6'] (p6') at (2,0.5) {$\bullet$};
	\node[label=0:7'] (p7') at (2,0) {$\bullet$};
	\draw (p1.center)--(p6'.center);
	\draw (p2.center)--(p5'.center);
	\draw (p3.center)--(p4'.center);
	\draw (p4.center)--(p3'.center);
	\draw (p5.center)--(p2'.center);
	\draw (p6.center)--(p1'.center);
	\draw (p1.center)--(p1'.center);
	\draw (p1.center)--(p5'.center);
	\draw (p3.center)--(p3'.center);
	\draw (p5.center)--(p1'.center);
	\draw (p2.center)--(p1'.center);
	\draw (p2.center)--(p4'.center);
	\draw (p4.center)--(p2'.center);
	\draw (p1.center)--(p2'.center);
	\draw (p1.center)--(p4'.center);
	\draw (p4.center)--(p1'.center);
	\draw (p3.center)--(p1'.center);
	\draw (p3.center)--(p2'.center);
\end{tikzpicture}
\raisebox{1.4cm}{ $\rightarrow$ }
\begin{tikzpicture}
	\node[label=180:1] (p1) at (0,3) {$\bullet$};
	\node[label=180:2] (p2) at (0,2) {$\bullet$};
	\node[label=180:3] (p3) at (0,1) {$\bullet$};
	\node[label=180:4] (p4) at (0,0) {$\bullet$};
	\node[label=0:1'] (p1') at (2,3) {$\bullet$};
	\node[label=0:2'] (p2') at (2,2) {$\bullet$};
	\node[label=0:3'] (p3') at (2,1) {$\bullet$};
	\node[label=0:4'] (p4') at (2,0) {$\bullet$};
	\draw (p3.center)--(p4'.center);
	\draw (p4.center)--(p3'.center);
	\draw (p1.center)--(p1'.center);
	\draw (p3.center)--(p3'.center);
	\draw (p2.center)--(p1'.center);
	\draw (p2.center)--(p4'.center);
	\draw (p4.center)--(p2'.center);
	\draw (p1.center)--(p2'.center);
	\draw (p1.center)--(p4'.center);
	\draw (p4.center)--(p1'.center);
	\draw (p3.center)--(p1'.center);
	\draw (p3.center)--(p2'.center);
\end{tikzpicture}
\caption{A $2$-triangulation (left), its bipartization (center) and the reduced version of the latter (right).}
\label{fig:bipartization}
\end{figure}
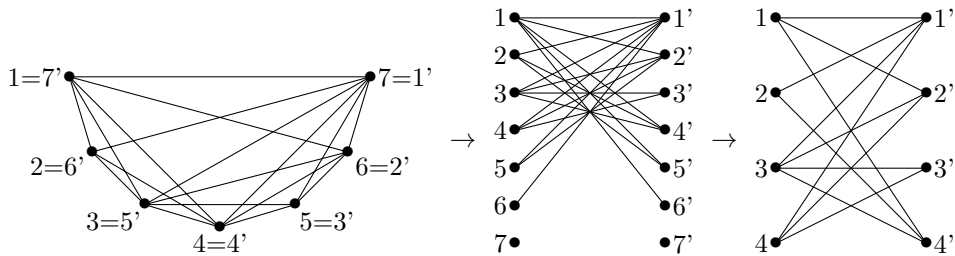

In this paper, we obtain results relating bipartizations of multitriangulations with an algebraic matroid, in the line of the already mentioned ones about Pfaffians, but using a projected determinantal variety instead of the Pfaffian variety. This allows us to prove Theorem \ref{thm:triangbases}.

The relation between bipartite graphs and the algebraic matroid of the determinantal variety is already pointed out in \cite{KTT} in relation to a different problem, that of \textit{matrix completion}. This problem asks the following question: given a matrix of which we only know some entries, and knowing that its rank is at most $k$, can the matrix be completed? If so, in how many ways (only one, finitely many, or a continuum of solutions)?

Concretely, this is what algebraic matroids tell us about matrix completion:
\begin{theorem}[{\cite[Theorem 2.14]{CreSan:Pfaffians}}]
	\label{thm:algebraic}
	Let $\K$ be an algebraically closed field, $I\subset \K[x_1,\dots,x_N]$ a prime ideal and $V$ its algebraic variety. For each $S\subset [N]$
	denote by $\pi_S : \K^{[N]}\to \K^S$ the coordinate projection to $S$. Then:
	\begin{enumerate}
		\item $S$ is independent in the algebraic matroid of $I$ if and only if $\pi_S(V)$ is  Zariski dense in $\K^S$.
		\item The rank of $S$ is equal to the dimension of $\pi_S(V)$.
		\item $S$ is spanning if and only if  $\pi_S$ is finite-to-one: for every $x\in \K^S$ the fiber $\pi_S^{-1}(x)$ is finite (perhaps empty).
	\end{enumerate}
\end{theorem}

See \cite{Rosen,RST} for more in algebraic matroids and matrix completion.

Reduced bipartizations of $k$-triangulations have $k(2n-2k-1)-k(k+1)=2kn-3k^2-2k$ edges, which is exactly the rank of the bipartite hyperconnectivity matroid of $2(n-k-1)$ vertices. This coincidence led us to conjecture that the rows of the bipartite hyperconnectivity matrix for generic points in dimension $k$ realize $\OvAss{k}{n}$ as a polytope, or at least as a complete fan.

For any given choice of points $p_1,\dots,p_{n-k-1},p_1',\dots,p_{n-k-1}' \in \R^k$ in general position, the rows of their bipartite hyperconnectivity matrix give a real vector configuration $\VV =\{p_{ij}\}_{i,j}$ of rank $2kn-3k^2-2k$. The question then is whether using those vectors as generators makes $\OvAss{k}{n}$ be a fan, and whether this fan is polytopal. Here, by  bipartite hyperconnectivity matroid we mean the generic $k$-hyperconnectivity matroid of the complete bipartite graph. 

All the realizations of $\OvAss{k}{n}$ that we construct use this strategy for positions  of the points along the \emph{moment curve} $\{(1,t,t^2,\dots,t^{k-1})\in\R^k:t\in \R\}$. One of the reasons to restrict our search to the moment curve is that in our previous paper~\cite{CreSan:moment} we show that, for points along the moment curve, the vector configuration obtained with hyperconnectivity coincides (modulo linear isomorphism) with configurations coming from the other two forms of rigidity: bar-and-joint rigidity along the moment curve and cofactor rigidity along the parabola. This makes our realizations more ``natural'', since they can be interpreted in the three versions of rigidity. Another reason is that being in the moment curve has some implications about the signs in the oriented matroid (Lemma \ref{lemma:chsign}).

Since our points lie in the moment curve, we can refer to each point $(1,t,\dots, t^{k-1})$ via its parameter $t$.

To realize the fan, a partial result is that the multitriangulations are bases: it says that at least the individual cones have the right dimension and are simplicial.

\begin{theorem}\label{thm:triangbases}
	Reduced bipartizations of $k$-triangulations are bases in the generic bipartite $k$-hyperconnectivity matroid, and independent in the generic $k$-rigidity matroid.
\end{theorem}

\begin{remark}
	In this theorem it is important to consider the bipartite $k$-hyper\-connectivity because the general $k$-hyperconnectivity has a higher rank $2k(n-k-1)-\binom{k+1}2$. Reduced bipartizations cannot be bases of that  matroid.
\end{remark}

By Corollary \ref{bipfree}, this theorem implies a previous result of ours:

\begin{theorem}[{\cite[Corollary 2.17]{CreSan:Pfaffians}}]
	\label{thm:hyperconnectivity}
	$k$-triangulations are bases in the generic hyperconnectivity matroid of dimension $2k$.
\end{theorem} 

First, we pose the conjecture that positions along the moment curve realizing $\OvAss{k}{n}$ as a basis collection exist for every $k$ and $n$:

\begin{conjecture}
	\label{conj:rigid}
	Reduced bipartized $k$-triangulations of the $n$-gon are bases in the bar-and-joint rigidity matroid of generic points along the moment curve in dimension $k$.
\end{conjecture}

This conjecture is a stronger version of Theorem \ref{thm:triangbases}. It would imply the same for generic bar-and-joint rigidity. The case $k=2$ is a consequence of Theorem \ref{thm:triangbases}, because in dimension 2 hyperconnectivity coincides with polynomial rigidity.

In fact, our experiments make us believe that in this case the  word ``generic'' can be changed to ``arbitrary'': all the positions along the moment curve that we have tried realize $\OvAss{k}{n}$ at least as a fan.

\begin{conjecture}
	\label{conj:rigid-k=2}
	Reduced bipartized $2$-triangulations are bases in the bipartite bar-and-joint rigidity matroid of arbitrary (distinct) points along the parabola.
\end{conjecture}

\begin{conjecture}\label{conj:fan}
	Every choice of ordered parameters $t_1<t_2<\dots < t_{n-3}, t'_{1}<\dots  <t'_{n-3} \in \R$
	gives a configuration in the moment curve that realizes $\OvAss{2}{n}$ as a fan.
\end{conjecture}

Observe that, by continuity, Conjecture~\ref{conj:rigid-k=2} implies that if some positions realize $\OvAss{2}{n}$ as a fan then \emph{all} positions do.

However, in the case $k\ge 3$, we show that some positions do not realize $\OvAss{k}{2k+3}$ as a fan (realizing $\OvAss{k}{n}$ for $n<2k+3$ is sort of trivial), and not even as a basis collection:

\begin{theorem}
	\label{thm:Desargues}
	The graph $K_9 - \{16, 37, 49\}$ is a $3$-triangulation of the $n$-gon, but its reduced bipartization is dependent in the bar-and-joint rigidity matroid in the moment curve if the cross-ratio between the hyperplanes $(12,23;24,25)$ equals $(2'4',2'3';1'2',2'5')$.
	This occurs, for example, if we take points along the moment curve with $\bt=(1,3,4,5,7,1,3,4,5,7)$.
\end{theorem}

This shows that Conjecture~\ref{conj:rigid-k=2} fails for $k\ge 3$, but we prove that it fails in the worst possible way.
We consider this our second main result:

\begin{theorem}\label{thm:2k+6}
	If $k=3$ and $n\ge 12$, or $k\ge 4$ and $n\ge 2k+4$, then no choice of points $\bt\in \R^{2(n-k-1)}$ in the moment curve makes the bar-and-joint rigidity matrix $\PP_k(\bt)$ 
	realize the $k$-associahedron $\OvAss{k}{n}$ as a fan.
\end{theorem}

Note that these two results are similar to the corresponding ones in \cite{CreSan:Multi}. In the second one, however, the obstruction proved here is stronger than the one from  \cite{CreSan:Multi}. There the case $(4,13)$ is explicitly realized as a fan via cofactor rigidity, here we prove that it cannot be realized with bipartite rigidity (along the moment curve).

We have also some experimental results:
\begin{theorem}  With bipartite rigidity with points in the moment curve,
	\begin{enumerate}
		\item for $k=2$, it is possible to realize the polytope for $n\le 10$ and the fan for $n\le 13$.
		\item for $k=3$, it is possible to realize the fan for $n\le 11$.
	\end{enumerate}
\end{theorem}

This theorem, together with the previous one, implies that the fan can be realized for $k\ge 3$ if and only if $n\le \max\{2k+3,11\}$.

To summarize, this new form of rigidity is similar to the  one used in \cite{CreSan:Multi}. The main difference lies in the case $(4,12)$ and $(4,13)$, which are realized as a fan in  \cite{CreSan:Multi} while are proved to be non-realizable with bipartite rigidity.
However, the main open question both here and in \cite{CreSan:Multi} is the case $k=2$. Although we have not been able to solve it  (Conjecture~\ref{conj:fan}), bipartite rigidity has a significant advantage here: we need to use $2$-dimensional rigidity instead of the $2$-dimensional one used in   \cite{CreSan:Multi}.

\subsection*{Structure of the paper}
In section \ref{sec:multi} we introduce all necessary background on multitriangulations. In section \ref{sec:hyper} we describe hyperconnectivity and tell about its properties, especially when restricted to bipartite graphs.

Section \ref{sec:hyper-bip} puts everything together and applies ideas of algebraic matroids and matrix completion to bipartizations of multitriangulations. Then we add tropical geometry to the picture, finding a relation between the multiassociahedron and this projected determinantal variety.

Finally, Section \ref{sec:real} analyzes the realizability of the multiassociahedron taking as vectors the rows of the bipartite hyperconnectivity matrix with points in the moment curve, proving the mentioned (theoretical and experimental) results.

\section{Multitriangulations}
\label{sec:multitriangs}
\label{sec:multi}

Let us recall in detail the definition of the $k$-associahedron. As mentioned in the introduction, it is a simplicial complex with vertex set 
\[
\bnn:=\{\{i,j\}: i,j \in [n], i<j\}.
\]

\begin{definition}
	Two disjoint elements $\{i,j\},\{k,l\}\in \bnn$, with $i<j$ and $k<l$, of $\binom{[n]}2$ \emph{cross} if $i<k<j<l$ or $k<i<l<j$. That is, if they cross when seen as diagonals of a convex $n$-gon.
	
	A \emph{$k$-crossing} is a subset of $k$ elements of $\binom{[n]}2$ such that every pair  cross. A subset of $\binom{[n]}2$ is \emph{$(k+1)$-free} if it doesn't contain any $(k+1)$-crossing. A \emph{$k$-triangulation} is a maximal $(k+1)$-free set.
\end{definition}

Observe that whether two pairs $\{i,j\},\{k,l\}\in \bnn$ cross is a purely combinatorial concept, but it captures the idea that the corresponding diagonals of a convex $n$-gon  geometrically cross.

The \emph{length} of an edge $\{i,j\}\in\bnn$, is $\min\{|j-i|, n- |j-i|\}$. That is, the distance from $i$ to $j$ measured cyclically in $[n]$.
Edges of length at most $k$ cannot participate in any $k+1$-crossing and, hence, all of them lie in every $k$-triangulation. We call edges of length at least $k+1$ \emph{relevant} and those of length at most $k-1$ \emph{irrelevant}. The ``almost relevant'' edges, those of length $k$, are called \emph{boundary edges} and, although they lie in all $k$-triangulations, they still play an important role in the theory.

By definition, $(k+1)$-free subsets form an abstract simplicial complex on the vertex set $\binom{[n]}2$, whose facets are the $k$-triangulations and whose minimal non-faces are the $(k+1)$-crossings. We denote it $\Ass{k}{n}$. Since the $kn$ irrelevant and boundary edges lie in every facet, it makes sense to consider also the \emph{reduced complex} $\OvAss{k}{n}$. Technically speaking, we have that $\Ass{k}{n}$ is the join of $\OvAss{k}{n}$ with the irrelevant face (the face consisting of irrelevant and boundary edges).

Multitriangulations were studied (under a different name) by Capoyleas and Pach~\cite{cp-92}, who showed that no $(k+1)$-free subset has more than $k(2n-2k-1)$ edges. That is, the complex $\Ass{k}{n}$ has dimension $k(2n-2k-1)-1$, hence $\OvAss{k}{n}$ has dimension $k(n-2k-1)-1$.
The main result about $\OvAss{k}{n}$ for the purposes of this paper is the following theorem of Jonsson:

\begin{theorem}[Jonsson~\cite{Jonsson}]
	\label{thm:sphere}
	$\OvAss{k}{n}$ is a shellable sphere of dimension $k(n-2k-1)-1$.
\end{theorem}

The following lemma shows that the realizability question we want to look at is \emph{monotone}; if we have a realization of $\OvAss{k}{n}$ then we also have it for all $\OvAss{k'}{n'}$ with $k'\le k$ and $n'-2k'\le n-2k$.

\begin{lemma}[{Monotonicity \cite[Lemma 2.3]{CreSan:Multi}}]
	\label{lemma:monotone}
	Let $n \ge 2k+1$. 
	Then, both $\OvAss{k}{n}$ and $\OvAss{k-1}{n-1}$ appear as links in $ \OvAss{k}{n+1}$. More precisely:
	\begin{enumerate}
		\item 
		$\OvAss{k}{n} = \link_{\OvAss{k}{n+1}}(B_{n+1})$, 
		where $B_{n+1}:= \{\{n-k+i,i\}: i=1,\dots,k\}\in \bnn$ is the set of edges of length $k+1$ leaving $n+1$ in their short side.
		\item
		$ \OvAss{k-1}{n-1} \cong \link_{\OvAss{k}{n+1}}(E_{n+1})$,
		where $E_{n+1}=\{\{i,n+1\}: i \in [k+1,n-k]\}\in \binom{[n+1]}{2}$ is the set of relevant edges using $n+1$.
	\end{enumerate}
\end{lemma}

Being a sphere (more precisely, being a pseudo-manifold) has the following important consequence:

\begin{proposition}[Flips]
	\label{prop:flips}
	For every relevant edge $f$ of a $k$-triangulation $T$ there is a unique edge $e\in\binom{[n]}2$ such that 
	\[
	T\triangle\{e,f\} := T\setminus \{f\}\cup\{e\}
	\]
	is another $k$-triangulation.
\end{proposition}

We call the operation that goes from $T$ to $T\triangle\{e,f\}$ a \emph{flip}.

The next result is about the link of a face $F$ in $\OvAss{k}{n}$. Remember that the link of a face $F$ in a simplicial complex $\Delta$ is
\[
	\link_\Delta(F) := \{ G\in \Delta: G\cap F=\emptyset, G\cup F\in \Delta\} = \{\sigma\setminus F: \sigma\in \Delta\}. 
\]
In a shellable $d$-sphere the link of any face of dimension $d'$ is a shellable $d-d'-1$-sphere.

\begin{proposition}[{\cite[Corollary 2.9]{CreSan:Multi}}]
	\label{prop:cycles}
	All links of dimension one in $\OvAss{k}{n}$ are cycles of length at most five.
\end{proposition}

With this, we can state a necessary and sufficient condition for a vertex configuration (an assignment of vectors to the vertices of $\Ass{k}{n}$) to realize the multiassociahedron as a complete fan:

\begin{theorem}[{\cite[Corollary 2.13]{CreSan:Multi}, see also \cite[Corollary 4.5.20]{triangbook}}]
	\label{thm:fan}
	Let $\VV=\{v_{ij}\}_{\{i,j\}\in \bnn} \subset \R^{k(2n-2k-1)}$ be a vector configuration. $\VV$ embeds $\OvAss{k}{n}$ as a complete fan in $\R^{k(n-2k-1)}$ if and only if it satisfies the following properties:
	\begin{enumerate}
		\item (Basis collection) For every facet ($k$-triangulation) $T$, the vectors $\{v_{ij}: \{i,j\}\in T\}$ are a linear basis. 
		\item (Interior Cocircuit Property, ICoP) For every flip between two $k$-triangulations $T_1$ and $T_2$, the unique linear dependence among the vectors $\{v_{ij}: \{i,j\}\in T_1\cup T_2\}$ has the same sign for the two elements involved in the flip (the unique elements in $T_1\setminus T_2$ and $T_2\setminus T_1$).
		\item (Elementary cycles of length $5$) Given a codimension 2 face $\rho$ whose link $Z$ is a cycle of length five (see Proposition \ref{prop:cycles}), there are three consecutive elements $i_1, i_2, i_3\in Z$ such that the unique linear dependence among the vectors $\{v_{i}: i\in \rho\cup\{i_1, i_2, i_3\}\}$ has opposite sign for $i_2$ than the sign it takes for $i_1$ and $i_3$.
	\end{enumerate}
\end{theorem}

\section{Bipartite hyperconnectivity}\label{sec:hyper}
\subsection{Definition of the matroid}

Let $\p=(p_1,\dots,p_n)$ be a configuration of $n$ points in $\R^d$, labelled by $[n]$.%
\footnote{
	By a \emph{configuration} we mean an ordered set of points or vectors, usually labelled by the first $n$ positive integers.
	For this reason we write $\p$ as a vector rather than a set.
} 
Their \emph{hyperconnectivity matrix} is the following $\binom{n}{2}\times nd$
matrix:
\begin{equation}\label{H}
H(\p):=
\begin{pmatrix}
p_2 & -p_1 & 0 & \dots & 0 & 0 \\
p_3 & 0 & -p_1 & \dots & 0 & 0 \\
\vdots & \vdots & \vdots & & \vdots & \vdots \\
p_n & 0 & 0 & \dots & 0 & -p_1 \\
0 & p_3 & -p_2 & \dots & 0 & 0 \\
\vdots & \vdots & \vdots & & \vdots & \vdots \\
0 & 0 & 0 & \dots & p_n & -p_{n-1}
\end{pmatrix}.
\end{equation}
Since there is a row of the matrix for each pair $\{i,j\}\in \bnn$, rows can be considered labeled by edges in the complete graph $K_n$. The matrix is a sort of ``directed incidence matrix'' of $K_n$, except instead of having one column for each vertex $i\in[n]$ we have a block of $d$ columns, and instead of putting a single $+1$ and $-1$ in the row of edge  $\{i,j\}$ we put the $d$-dimensional vectors $p_j$ and $-p_i$ (if we put instead $p_i-p_j$ and $p_j-p_i$, the result would be the bar-and-joint rigidity matrix for the same points). 

An important property of $H(\p)$ (Lemma 3.4 in \cite{Kalai}) is that if the points $\p$ linearly span $\R^d$ then the rank of $H(\p)$ equals
\begin{equation}
	\label{eq:rank}
	\begin{cases}
		\binom{n}{2} & \text{if}\ n\leq d+1, \\
		dn-\binom{d+1}{2} & \text{if}\ n\geq d.
	\end{cases} 
\end{equation}
(Observe that the two formulas give the same result for $n\in\{d,d+1\}$.)
If the points span an $r$-dimensional linear subspace, the same formulas hold with $r$ substituted for $d$.

A pair $(G, \p)$ where $G$ is a graph on $n$ vertices and $\p$ is a set of $n$ points in $\R^n$ (positions for the vertices) is usually called a framework. For any  $E\subset \bnn$ we denote by $H(\p)|_E$ the restriction of $H(\p)$ to the rows or elements indexed by $E$.

\begin{definition}[The hyperconnectivity matroid]
	Let $E\subset \bnn$ be a subset of edges of $K_n$ (equivalently, of rows of $H(\p)$).  We say that $E$, or the corresponding subgraph of $K_n$,  is \emph{self-stress-free} or \emph{independent} if the rows of $H(\p)|_E$ are linearly independent. 
\end{definition}

Put differently, self-stress-free and rigid graphs are, respectively, the independent and spanning sets in the linear matroid of rows of $H(\p)$. We call this matroid the \textit{hyperconnectivity matroid} of $\p$ and denote it $\HH(\p)$. It is a matroid with ground set $\bnn$ and, for points affinely spanning $\R^d$, of rank given by Equation~\eqref{eq:rank}.
Let us remark that, although rigidity theory usually deals only with $\HH(\p)$ as an (unoriented) matroid, its definition as the linear matroid of a configuration of real vectors produces in fact an \emph{oriented matroid}.

The matroid $\HH(\p)$ is invariant under linear transformation. (More generally, it is invariant under projective transformation in $\R\proj^{d-1}$ as a quotient of $\R^d\setminus \{0\}$).
We say that the points chosen are in \emph{general position} for $\HH$ if no $d$ of them lie in an linear hyperplane. This implies that the matroid has the rank stated in Equation~\eqref{eq:rank} and that every copy of the graph $K_{d+2}$ is a circuit.  Nguyen~\cite{Nguyen} showed that the matroids on $\bnn$ with these properties are exactly the  \textit{abstract rigidity matroids} introduced by Graver in \cite{Graver}; other two matroids with the same characteristics are the classical (bar-and-joint) rigidity matroid and the cofactor rigidity matroid.
See, for example, \cite{GSS,Whiteley} for more information on rigidity matrices and their matroids.

Clearly, for each choice of the ``dimension'' $d$ there is a unique most free matroid that can be obtained, realized by sufficiently generic choices of the points. We call this the \emph{generic hyperconnectivity matroid of dimension $d$ on $n$ points}, and denote it $\HH_d(n)$. 
(Observe, however, that this generic matroid may stratify into several different generic oriented matroids; this is important for us since we will be concerned with the signs of circuits, by Theorem \ref{thm:fan}).

We can restrict the hyperconnectivity matroid to bipartite graphs \cite[Sect. 6]{Kalai}. The resulting matrix has the same form than $H(p)$, but in the points in the left side there is always $p_j'$ for the neighbours $j'$ in the right side, and for the points in the right side there is always $-p_i$ (or equivalently $p_i$, because we can change the sign of the columns) for the neighbours $i$.

The paper \cite{KNN} defines a more general concept of ``$(a,b)$-rigidity matroid'' where the points at one side of the graph are in one dimension $a$ and the points at the other side in a different dimension $b$. Hence, the matrix has $n_1n_2$ rows and $n_1b+n_2a$ columns. We are mainly not interested in this situation, but it will appear as an intermediate step in a proof.

In \cite{CreSan:moment} we prove that, for bipartite graphs, hyperconnectivity is a special case of bar-and-joint rigidity, more concretely:
\begin{theorem}[{\cite[Theorem 4.4]{CreSan:moment}}]
	\label{hyper-bipartite}
	Let $G=(V,E)$ be a bipartite graph with vertex partition $V=X\dot\cup Y$, and let $\{p_i : i\in V\} \subset \R^{d-1}$ be positions for its vertices. Then, the following two $d$-dimensional rigidity matroids coincide:
	\begin{enumerate}
		\item the bar-and-joint  matroid of the points $\{(p_i, 0): i\in X\} \cup \{(p_j, 1):j\in Y\}$.
		\item the hyperconnectivity matroid of the points $\{(p_i, 1): i\in X \cup Y\}$.
	\end{enumerate}
\end{theorem}

In the same paper we prove that hyperconnectivity coincides with other two common theories, bar-and-joint and cofactor rigidity, when the points are chosen along the moment curve (for bar-and-joint and hyperconnectivity) and the parabola (for cofactor). More precisely:

\begin{theorem}[{\cite[Theorem 1.1]{CreSan:moment}, \cite[Theorem 2.17]{CreSan:Multi}}] 
	\label{thm:3matroids}
	Let $t_1< \dots < t_n\in \R$ be real parameters. Let 
	\[
	p_i=(1,t_i,\dots,t_i^{d-1})\in \R^d, \quad
	p'_i=(t_i,t_i^2,\dots,t_i^d)\in \R^d, \quad
	q_i=(t_i,t_i^2)\in \R^2.
	\]
	Then, the matrix $H(p_1,\dots,p_n)$, the bar-and-joint rigidity matrix $R(p'_1,\dots,p'_n)$ and the cofactor rigidity matrix $C_d(q_1,\dots,q_n)$ can be obtained from one another 
	multiplying on the right by a regular matrix and then multiplying rows by some positive scalars.
	
	In particular, the rows of the three matrices define the same oriented matroid. 
\end{theorem}

\begin{definition}[Polynomial rigidity]
	\label{def:poly}
	We call the matrix $H(p_1,\dots,p_n)$ in the statement of Theorem~\ref{thm:3matroids} the \emph{polynomial $d$-rigidity matrix with parameters $t_1,\dots,t_n$}. We denote it $P_d(t_1,\dots,t_n)$, and denote $\PP_d(t_1,\dots,t_n)$ the corresponding matroid.
\end{definition}
For a bipartite graph, that means $(1,t_j',(t_j')^2,\ldots,(t_j')^{d-1})$ in the $i$-th part of the row $(i,j)$, and $(1,t_i,t_i^2,\ldots,t_i^{d-1})$ for the $j'$-th part.

Among the polynomial rigidity matroids $\PP_d(t_1,\dots,t_n)$ there is again one that is the most free, obtained with a sufficiently generic choice of the $t_i$. We denote it $\PP_d(n)$ and call it the \emph{generic polynomial $d$-rigidity matroid on $n$ points}. Theorem~\ref{thm:3matroids} implies that we can regard $\PP_d(n)$ as capturing generic bar-and-joint rigidity along the moment curve, generic hyperconnectivity along the moment curve, or generic cofactor rigidity on a conic.

It is known that $\PP_2(n)=\HH_2(n)$ (\cite[Theorem 2.2]{CreSan:moment}), but for higher dimensions, the equality of the matroids is an open problem.

\subsection{Properties of bipartite hyperconnectivity and polynomial rigidity}

The first property relates linear dependences in the hyperconnectivity matrix to linear dependences between the vectors.
\begin{lemma}\label{lemma:extprod}
	Given a framework $(V,E,p)$, the exterior product of two linear dependences between the points in $p$ gives a dependence in $H_d$.
\end{lemma}
\begin{proof}
	By the definition of the matrix,
	\[\sum_{j<i}(c_jc'_i-c'_jc_i)(-p_j)+\sum_{j>i}(c_ic'_j-c'_ic_j)p_j=\sum_{j\ne i}(c_ic'_j-c'_ic_j)p_j=\]
	\[c_i\sum_{j\ne i}c'_jp_j-c'_i\sum_{j\ne i}c_jp_j=-c_ic'_ip_i+c'_ic_ip_i=0\]
\end{proof}

\begin{corollary}
	The $K_{d+1,d+1}$ graph on the nodes $p_1,p_2,\ldots$, $p_{d+1}$ and $p_1',p_2',\ldots$, $p_{d+1}'$ is a circuit in the hyperconnectivity matrix, with the coefficients given as the exterior product of the two linear dependences in the $p_i$ and the $p_i'$, which in this case gives just a tensor product.
\end{corollary} 

Although the unrestricted hyperconnectivity matroid has the same rank as the bar-and-joint and cofactor matroids, there are no bipartite bases of this matroid (unlike both bar-and-joint and cofactor matroids, for which $K_{d+1,\binom{d+1}2}$ is a basis). So if we restrict it to bipartite graphs, the rank falls.

\begin{theorem}  
	\label{thm:complete_bipartite}
	Let $G=K_{n_1,n_2}$ be a complete bipartite graph. Assuming that the positions chosen for each part of vertices are linearly independent, 
	the rank of $G$ in the $d$-hyperconnectivity matroid  equals:
	\begin{enumerate}
		\item $n_1n_2$ (that is, $G$ is independent) if $\min\{n_1,n_2\} \le d$.
		\item $dn-d^2$ if $\min\{n_1,n_2\} \ge d$,  where $n=n_1+n_2$ is the number of vertices.
	\end{enumerate}
\end{theorem}

\begin{proof}
	For part (1), the edges of a vertex of degree $d$ or less are independent of the rest, in an arbitrary graph. This is a common property of all abstract rigidity matroids.

	For part (2), as the matrix has $n_1n_2$ rows, we just have to find $(n_1-d)(n_2-d)$ dependences among them; that is, vectors orthogonal to every column.  We consider columns as living in $\R^{n_1}\otimes \R^{n_2}$ and observe that every tensor product $l\otimes m$ of a linear dependence $l$ among the points on one side and a linear dependence $m$ among the points on the other side is such an orthogonal vector. That is, we have the tensor product of two linear spaces of dimensions $n_1-d$ and $n_2-d$, which indeed has dimension $(n_1-d)(n_2-d)$.
\end{proof}

\begin{remark}
	It is also easy to find $d^2$ linear dependences among the $nd$ columns. 
	Let $A=\{x_1,x_2,\ldots,x_{n_1}\}\subset \R^d, B=\{y_1,y_2,\ldots,y_{n_2}\}\subset \R^d$ be the positions for our vertices.  For each pair $(l,l') \in [d]^2$, let $U\in \R^{dn}$ be the vector with coordinates 
	\[u_{il}=x_{il'}, i\in[n_1]\]
	\[u'_{jl'}=-y_{jl}, j\in[n_2],\]
	where the $u_{il}$ are the coefficients for the first half of the columns and $u'_{jl}$ for the second half. The coefficients $u_{ih}$ or $u'_{jh'}$ with $h\ne l$ or $h'\ne l'$ are zero.
	Then, $U$ is the vector of coefficients in the sought dependence.
\end{remark}

\begin{remark}
	The rank computed in Theorem~\ref{thm:complete_bipartite} coincides with the rank of $G$ in the usual bar-and-joint rigidity for generic positions on the moment curve. This supports the conjecture that generic hyperconnectivity coincides with generic rigidity along the moment curve. Put differently, that for hyperconnectivity the moment curve is generic enough.
\end{remark}

The next result is a bipartite version of the linear invariance of hyperconnectivity:

\begin{proposition}
	\label{prop:invariance}
	Let $n_1,n_2\in\N$ and $\p=(p_1,\dots,p_{n_1},p_1',\ldots,p'_{n_2})$ be a vector configuration in $\R^d$.
	Then, 
	\begin{enumerate}
		\item The column-space of the $H_{K_{n_1n_2}}(\p)$, hence the oriented matroid of its rows, that is, the bipartite hyperconnectivity matroid, is invariant under a linear transformation of the points $(p_1,\dots,p_{n_1})$ in one side of the graph.
		\item The matroid $\HH_d(\p)$ is also invariant under rescaling (that is, multiplication by non-zero scalars) of the vectors $p_i$. If the scalars are all positive then the same holds for the oriented matroid.
	\end{enumerate}
\end{proposition}

\begin{proof}
	Both transformations consist in making linear operations in the matrix: for the first, multiplying at the right by
	\[\begin{pmatrix}
		I & 0 \\
		0 & M
	\end{pmatrix}\]
	where $I$ is the identity matrix with size $n_1d$ and $M$ is a block-diagonal matrix with $n_2$ blocks, all equal to the matrix of the linear transformation. As we are multiplying at the right, this does not change the column-space.
	
	For the second, as already noted in \cite[Lemma 2.1]{CreSan:moment}, if $p_i$ is rescaled to $\alpha_ip_i$, the effect in the matrix is multiplying the rows for the edges in the vertex $i$ by $\alpha_i$ and then dividing the columns for the vertex $i$ by the same scalar. If the scalars are positive, it does not affect the oriented matroid.
\end{proof}

We also need a bipartite version of the operation called \textit{coning}. In this operation, a new vertex is added to the graph and connected to all the rest, but obviously we cannot do this with a bipartite graph. However, there is another operation that works.

\begin{definition}
	Given a graph $G=([n_1]\cup [n_2]',E)$, the bipartite coning of $G$ is the graph with vertex set $[n_1+1]\cup[n_2+1]'$ where the two new vertices are joined to all the previous ones in the opposite side.
\end{definition}

It is proved in \cite[Lemma 3.12]{KNN} that the bipartite coning of $G$ is independent in the generic hyperconnectivity matroid in dimension $d+1$ if and only if $G$ is independent in dimension $d$. Our next result is more precise in two aspects: we deal with arbitrary, not necessarily generic, positions, and we include information about the \emph{oriented} matroid, not just the matroid.

\begin{theorem}[Coning]
For any vector configuration $V$, the contraction of the oriented matroid for $K_{n_1+1,n_2+1}$ in dimension $d+1$ by the edges in vertices $n_1+1,(n_2+1)'$ gives the oriented matroid for $K_{n_1,n_2}$ in dimension $d$, where the vector configuration is contracted by the vectors for $a$ and $b$.
\end{theorem}
\begin{proof}
	Let $G'$ be the bipartite coning of $G$. 
	Let us
represent the two vertex additions in the rigidity matrix as follows: Start with the $(d+1,d+1)$-rigidity matrix and make a linear transformation to send the new vertex $b$ to $(1,0,\ldots,0)$, that we will suppose to be in the right side. Now delete a block of columns in that side, which is redundant for the matroid because of the linear dependences in the columns, so that the rows for the edges in $b$ have only a nonzero entry. Thus, we can delete those rows and the associated columns, which are the first columns in all the blocks in the left side. The resulting matrix is the $(d+1,d)$-rigidity matrix for $G'\setminus b$ and the configuration where the right vectors have lost the first coordinate, that is, they have been contracted by $b$. Repeating this transformation with the other side, we obtain the result.
\end{proof}

We also want to see what happens with the oriented matroid of polynomial rigidity when the new vertex is not added the last, but in any other position.

\begin{corollary}\label{cor:signconing}
	Let $C$ be a bipartite circuit in $\PP_d$. The bipartite coning of $C$ where the vertices $i_0$ and $j_0'$ are added is a circuit in $\PP_d$ where the sign of an edge $(i,j)$ is multiplied by that of $(i-i_0)(j-j_0)$.
\end{corollary}
\begin{proof}
	As in the previous proof, we can do the coning in two steps, adding first the vertex at one side and then the other.
	
	Suppose we are adding $j_0'$ to get $G'$. The matrix after the addition is the polynomial rigidity matrix in dimension $d+1$ for $(t_1,\ldots,t_{n_1},t_1',\ldots,t_{n_2}')$. As we did before, move the vertex $j_0$ to $(1,0,\ldots,0)$. We can do this by translating all the parameters $t'$ by $-t_{j_0}'$, which is a linear transformation in the points. Now we have $t_{j_0}'=0$ and the sign of $t_j'$ is the same than $j-j_0$.
	
	Then, removing rows and columns, we end up with the $(d+1,d)$-rigidity matrix for $G'\setminus j_0'$ in the positions $(p_1,\ldots,p_{n_1},p_1'',\ldots,p_{n_2}'')$, where $p_j''=(t_j',(t_j')^2,\ldots,(t_j')^d)=t_j'p_j'$.
	
	Now, dividing the row for edge $(i,j)$ by $t_j'$ and then multiplying the columns in the block for $j$ by $t_j'$, we have the polynomial rigidity matrix for dimension $d+1$ in the left and $d$ in the right. Repeating this for $i_0$, we finally obtain dimension $d$.
\end{proof}

The following result is about the number of sign changes in the sequence of coefficients for the edges in a given vertex.

\begin{lemma}\label{lemma:chsign}
	Let $\lambda\in \R^{n_1n_2}$ be a linear dependence in the polynomial rigidity matroid and $v$ a vertex. The sequence $\{\lambda_{ij}\}_{1\le j\le n_2}$ of signs of the edges in $v$ change at least $d$ times.
\end{lemma}
\begin{proof}
	The sequence that we want to study is a linear dependence among the vectors $(1, t_j,\dots, t_j^{d-1})$ for an increasing sequence of $t_j$'s. Put differently, it is an affine dependence among the vertices of a cyclic $(d-1)$-polytope. It is well-known that the circuits in the cyclic polytope are alternating sequences with $d+1$ points, hence they have $d$ sign changes, and every dependence is a composition of circuits, hence it has at least the same number of changes.
\end{proof}

\begin{corollary}\label{cor:signKd+1}
	In that setting, the signs of the edges in a vertex of degree $d+1$ alternate. Moreover, for the circuit $K_{d+1,d+1}$, we have $\sign(\lambda_{ij})=(-1)^{i+j}$.
\end{corollary}

Our next result is an analogue of the Morgan-Scott obstruction \cite{MorganScott} and its signed version from \cite[Section 3.2]{CreSan:Multi}, but using hyperconnectivity instead of cofactor rigidity and the cube instead of the octahedron. Consider the bipartite graph obtained removing the perfect matching $11',22',33',44'$ from $K_{4,4}$ in the vertices $\{1,2,3,4\}$ and $\{1',2',3',4'\}$, that we suppose in that order in a line. This is the graph of the cube and it is a ``bipartite basis'' in the hyperconnectivity matroid in dimension 2 (it has one edge less than needed to be a basis, but none of the four diagonals can be added keeping independence).\footnote{This is a difference with bar-and-joint rigidity, in which a cube plus any edge is a basis.}

If we add the diagonal $11'$ to the cube, we get a circuit. In this circuit, all the signs are fixed except $11'$.

\begin{theorem}\label{cube}
	Let $G=K_{4,4} \setminus\{ 22',33',44'\}$. In the hyperconnectivity matroid $\HH_2$ (for positions in linear general position),
	the coefficient of the edge $11'$ in the circuit has the same sign  (respectively, opposite sign) as that of $12'$ if and only if the cross-ratio $(1,2;3,4)$ is less than (respectively greater than) the cross-ratio of $(1',2';3',4')$. In particular, the coefficient of $11'$ vanishes if and only if the two cross-ratios coincide.
\end{theorem}

\begin{proof}
	Let the position of the vertex $i$ be $x_i$ and that of $i'$ be $y_i$.
	
	To compute the circuit, we will cancel the edge $22'$ between the circuits in the graphs $\{1,2,3\}\times\{1',2',4'\}$ and $\{1,2,4\}\times\{1',2',3'\}$. Each one is a $K_{3,3}$ in which we know how to compute the signs: they are the exterior product of the two dependences. In this case, the dependences themselves have an easy expression: in $\{1,2,3\}$, for example, the coefficients are the distances $23$, $31$ and $12$, respectively.
	
	Applying this to the first graph and normalizing the coefficient of $22'$ to 1, we get the coefficient of $11'$ as
	\[\frac{(x_3-x_2)(y_4-y_2)}{(x_1-x_3)(y_1-y_4)}=\frac{(x_3-x_2)(y_4-y_2)}{(x_3-x_1)(y_4-y_1)}\]
	Doing the same in the second graph, we get analogously
	\[\frac{(x_4-x_2)(y_3-y_2)}{(x_4-x_1)(y_3-y_1)}\]
	The difference between these expressions gives the circuit we want. The sign of $31'$ is the same as in the first dependence (because this edge is not in the second one), and if we normalized $22'$ to 1, $31'$ will also be positive. So here we must have $21'$ negative and $12'$ positive.
	
	Multiplying the coefficient for $11'$ in the dependence by a positive factor, we get
	\[\frac{(y_3-y_1)(y_4-y_2)}{(y_4-y_1)(y_3-y_2)}-\frac{(x_3-x_1)(x_4-x_2)}{(x_4-x_1)(x_3-x_2)}=(1',2';3',4')-(1,2;3,4)\]
	as we wanted.
\end{proof}

Note that the cross-ratio  is the only projective invariant between four points in a projective line. Hence:

\begin{corollary}
	$G=K_{4,4} \setminus\{ 11', 22',33',44'\}$ (that is, the graph of a $3$-cube) is a circuit in $\HH_2$ if, and only if, the two sets of four points are projectively equivalent.
\end{corollary}

\section{Hyperconnectivity of bipartized multitriangulations}
\label{sec:hyper-bip}
\subsection{Bipartized multitriangulations and stack polyominoes}
As stated in the introduction, the reduced bipartization of a triangulation is a bipartite graph with $n-k-1$ vertices in each side and with $2kn-k(2k+1) - k(k+1) = 2kn-3k^2-2k = k(2n-2k-2) -k^2$ edges. This is the adequate number of edges for a basis in the bipartite hyperconnectivity matroid, so there are chances that it is a basis, as we will now show. Actually, we will work in a more general framework that will allow us to prove the independence of a larger class of graphs.

Consider the Cartesian product $[n_1]\times[n_2]$ as a poset, taking the increasing order in $[n_1]$ and $[n_2]$. Recall that an order ideal is a subset $S\subset [n_1]\times[n_2]$ so that for all $(a,b)\in S$ we have that $(a',b')\in S$ if $(a',b')\le(a,b)$, that is, $a'\le a$ and $b'\le b$. 

\begin{definition}
A \emph{Ferrers diagram} is an order ideal in the poset  $[n_1]\times[n_2]$.
\end{definition}

Ferrers diagrams are a particular case of \textit{stack polyominoes} \cite{Jonsson}.

In what follows, let $S$ be a Ferrers diagram.
Observe that this condition is equivalent to the existence of positions $x_1<\dots<x_{n_1}$ and $y_1> \dots >y_{n_2}$  of the vertices $a\in [n_1], b\in [n_2]$ of $K_{n_1,n_2}$ along a line such that $x_a < y_b$ if and only if $(x_a,y_b)\in S$. If we consider the elements of $[n_1]\times[n_2]$ as edges in a bipartite graph, and locate the vertex $a$ in the position $x_a$ and $b'$ in the position $y_b$, the edges in $S$ are exactly those going right from $a$ to $b'$.

For the reduced bipartization of a $k$-triangulation, $n_1=n_2=n-k-1$. An edge $(a,b)$ can appear in the bipartized graph if and only if $\{a,n+1-b\}$ is an edge in the original graph, that is, $a<n+1-b$. So the set of valid edges is
\[S_k(n):=\{(a,b):a,b\in[n-k-1], a+b\le n\}\]

For the graph in Figure \ref{fig:bipartization}, we can represent $S_2(7)$ as a Ferrers diagram, formed by the ``possible'' edges, which will contain the edges in the bipartized 2-triangulation itself. Here the edges in the graph are marked with x and the edges in $S_2(7)$ but not in the graph are marked with dots.
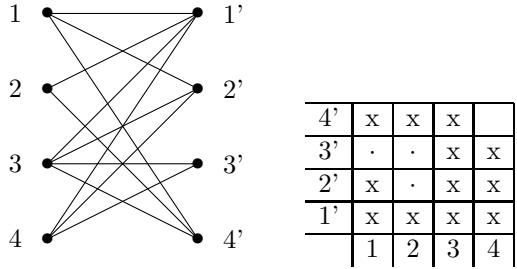
\begin{center}
\begin{tikzpicture}
	\node[label=180:1] (p1) at (0,3) {$\bullet$};
	\node[label=180:2] (p2) at (0,2) {$\bullet$};
	\node[label=180:3] (p3) at (0,1) {$\bullet$};
	\node[label=180:4] (p4) at (0,0) {$\bullet$};
	\node[label=0:1'] (p1') at (2,3) {$\bullet$};
	\node[label=0:2'] (p2') at (2,2) {$\bullet$};
	\node[label=0:3'] (p3') at (2,1) {$\bullet$};
	\node[label=0:4'] (p4') at (2,0) {$\bullet$};
	\draw (p3.center)--(p4'.center);
	\draw (p4.center)--(p3'.center);
	\draw (p1.center)--(p1'.center);
	\draw (p3.center)--(p3'.center);
	\draw (p2.center)--(p1'.center);
	\draw (p2.center)--(p4'.center);
	\draw (p4.center)--(p2'.center);
	\draw (p1.center)--(p2'.center);
	\draw (p1.center)--(p4'.center);
	\draw (p4.center)--(p1'.center);
	\draw (p3.center)--(p1'.center);
	\draw (p3.center)--(p2'.center);
\end{tikzpicture}\qquad
\begin{tabular}[b]{c|c|c|c|c|}\hline
	4' & x & x & x & \\ \hline
	3' & $\cdot$ & $\cdot$ & x & x \\ \hline
	2' & x & $\cdot$ & x & x \\ \hline
	1' & x & x & x & x \\ \hline
	 & 1 & 2 & 3 & 4 \\
\end{tabular}
\end{center}

Basically, the reason why the graph contains no 3-crossings, despite its appearance, is that the vertex $4$ should be at the right of $4'$, or, said otherwise, $(4,4)\notin S_2(7)$. This leads us to the following definition.

\begin{definition}
A $(k+1)$-crossing in $S$ is a set of $k+1$ incomparable elements whose supremum (in $[n_1]\times[n_2]$) belongs to $S$, that is, a set $\{(a_i,b_i)\}_{i=1,\dots,k+1} \subset S$ such that $a_i<a_{i+1}$ and $b_i>b_{i+1}$ for all $i$, and $(a_{k+1},b_1)\in S$. If we consider the elements of $S$ as edges in a bipartite graph with the vertices in a parabola and all edges going right, this is equivalent to say that the edges in question cross. A set, or bipartite graph, is $(k+1)$-free if it has no $(k+1)$-crossing. In this way, $(k+1)$-crossings and $(k+1)$-free sets generalize our previous definitions, that are obtained when $S=S_k(n)$, and a (bipartized) $k$-triangulation is a maximal $(k+1)$-free graph in $S_k(n)$.
\end{definition}

We can write $S$ as
\[S=\{(a,b):1\le a\le n_1,1\le b\le s_a\}\]
for adequate integers $n_2\ge s_1\ge s_2\ge\ldots\ge s_{n_1}\ge 1$.

There is a ``greedy'' subset $T$ that is a maximal subset without $(k+1)$-crossings, namely
\[
T_0:=\{(a,b)\in S: a\le k \text{ or } b\ge s_a-k+1  \}.
\]
Jonsson \cite[section 3]{Jonsson} showed the following:
\begin{theorem}\label{free}
For every Ferrers diagram $S$, all maximal $(k+1)$-free sets have exactly the same size as $T_0$.
\end{theorem}
Observe that the biggest possible size for $T_0$ is $k(n_1+n_2)-k^2$, attained when $(k,n_2), (n_1,k) \in S$; that is, $s_k=n_2$ and $s_{n_1}\ge k$. When this happens we will say that $S$ is \textit{$k$-full}. In what follows, we suppose that $S$ is $k$-full, because the elements of any row or column with less than $k$ vertices will not affect $(k+1)$-freeness. (Essentially, that is what we are doing when we reduce a bipartized graph.)

This result of Jonsson can be interpreted as saying that any induced subgraph of a $(k+1)$-free bipartite graph has the adequate number of edges to be itself independent. In fact, taking a subset of the vertices is equivalent to delete rows or columns of $S$, which will leave us with a Ferrers diagram of $[n_3]\times[n_4]$, for $n_3\le n_1$ and $n_4\le n_2$, to which we can apply the Theorem.

\begin{corollary}
	A maximal $3$-free bipartite graph plus any edge is a basis in the generic $2$-rigidity matroid.
\end{corollary}
\begin{proof}
	Any maximal $(k+1)$-free set has $2(n_1+n_2)-4$ edges, and any subgraph induced by $n_3+n_4$ vertices has at most $2(n_3+n_4)-4$ edges. Adding any edge to the graph, the resulting graph satisfies the so-called Laman condition, which is equivalent to generic rigidity in dimension $2$.
\end{proof}
This implies, of course, that $3$-free graphs are independent in generic rigidity in dimension $2$. It does not imply that they are independent in generic hyperconnectivity.

However, precisely that will be the main result of this section:
\begin{theorem}\label{thm:bipbases}
	$(k+1)$-free bipartite graphs defined in a set $S$ are independent in the bipartite hyperconnectivity matroid in dimension $k$. Hence, the maximal of these graphs (that have $k(n_1+n_2)-k^2$ edges) are bases.
\end{theorem}

Applying Theorem \ref{hyper-bipartite}, we get that these graphs are independent in the rigidity matroid with the positions of the vertices in two hyperplanes. This has as a consequence the generic rigidity of these graphs.

\begin{corollary}\label{cor:bipbases}
	$(k+1)$-free bipartite graphs defined in a set $S$ are independent in the bar-and-joint rigidity matroid in dimension $k$.
\end{corollary}

Theorem \ref{thm:bipbases} and Corollary \ref{cor:bipbases} give Theorem \ref{thm:triangbases} as a particular case. It would be interesting to relate it to the original multitriangulations, but we know no analogue to Corollary \ref{bipfree} that relates the bar-and-joint rigidity of a graph in dimension $2k$ and that of its bipartization in dimension $k$.

\subsection{Bipartite hyperconnectivity as an algebraic matroid}
To take up the proof of Theorem \ref{thm:bipbases}, we need some algebraic background. This is similar to Section 2 of \cite{CreSan:Pfaffians}, but as we are now dealing with bipartite graphs, our matrices have no symmetry, and we have something more similar to the determinantal variety than to the Pfaffian variety. Also, we are only interested in a part of the matrix, that corresponds to the Ferrers diagram $S$.

Let $\M_k(n_1,n_2)\subset \R^{n_1\times n_2}$ be the $k$-th \emph{determinantal variety}, consisting of the $n_1\times n_2$ matrices with rank at most $k$. Let $I_k(n_1,n_2)$ the corresponding ideal, which is generated by the $k+1$-size minors.

\begin{theorem}\label{thm:hyper-algebraic}
	The bipartite generic hiperconnectivity matroid for $n_1+n_2$ vertices in dimension $k$ coincides with the algebraic matroid of $\M_k(n_1,n_2)$.
\end{theorem}

\begin{proof}
	Recall that if a variety $V\subset \R^N$ is parametrized as the image of a polynomial map $f: \R^M\to \R^N$ then the algebraic matroid of $V$ coincides with the linear matroid of the Jacobian of $f$ at a generic point \cite[Proposition 2.5]{Rosen}.
	
	In our case, $\M_k(n_1,n_2)$ is the image of
	\[T
	:\R^{n_1\times k}\times \R^{k\times n_2}\rightarrow\R^{n_1\times n_2}
	\]
	where $T$ is the matrix product. The entries in the Jacobian of $T$ are as follows. Let $(A,B)\in \R^{n_1\times k}\times \R^{k\times n_2}$ and denote as $a_1,\dots,a_{n_1}$ the rows of $A$ and $b_1,\dots,b_{n_2}$ the columns of $B$. All of them are vectors in $\R^k$. Then:
	\[
	\frac{\partial T(A,B)_{ij}}{\partial a_{rs}}=\delta_{ir}b_{sj},\qquad
	\frac{\partial T(A,B)_{ij}}{\partial b_{st}}=\delta_{jt}a_{is},
	\]
	where $\delta_{ij}=1$ if $i=j$ and 0 otherwise.
	That is, the row for the element $ij$ is exactly the corresponding row of the hyperconnectivity matroid, with the coordinates of $A$ sorted by rows and the coordinates of $B$ by columns.
\end{proof}

Now let $\M_k(S)$ be the restriction of $\M_k(n_1,n_2)$ to a subset $S$ of coordinates that is a Ferrers diagram, and $I_k(S)$ the ideal of this variety (which consists in the $(k+1)$-size minors contained in $S$). Then, the algebraic matroid of $\M_k(S)$, for any $S$, is precisely our bipartite hiperconnectivity matroid, and what we want to prove is that the graphs lacking certain sets of edges (in this case, $(k+1)$-crossings), are independent in the matroid.

We also define $\M_k(n):=\M_k(S_k(n))$ and $I_k(n):=I_k(S_k(n))$. This result is the core of the proofs in \cite{CreSan:Pfaffians} and will be useful again in this setting.

\begin{proposition}\label{indep}
	Let $I\subset K[x_1,\ldots,x_N]$ be a prime ideal, $v$ a weight vector, $\mathcal{B}$ a Gr\"obner basis for $I$ with respect to $v$ and $\mathcal{F}\subset 2^{[N]}$ the support of the leading terms (with respect to $v$) of the polynomials in $\mathcal{B}$. Then, all subsets of $[N]$ that do not contain any element of $\mathcal{F}$ are independent in the algebraic matroid of $I$.
\end{proposition}
\begin{proof}
	Suppose that $T$ is dependent in the algebraic matroid. Then there is a polynomial $f\in I$ using only variables in $T$. In particular, $T$ contains the support of the leading term of $f$. As $\mathcal{B}$ is a Gr\"obner basis, this leading term is multiple of a leading term of a polynomial in $\mathcal{B}$, so $T$ contains an element of $\mathcal{F}$, as we wanted.
\end{proof}

That is, to prove independence of the $(k+1)$-free sets, we need weight vectors for which the minors are a Gr\"obner basis of the ideal they generate and that select in each minor the term with the $(k+1)$-crossing. For the standard set $S_k(n)$, it turns out that our ideal is a relabelling of an initial ideal of the Pfaffian ideal.

At this point we need to recall some concepts about Pfaffians. A \emph{Pfaffian of degree $k$} is the square root of  the determinant of an antisymmetric matrix $M$ of size $2k$.
Considering the entries of  $M$ as indeterminates (over a certain field $\K$), the Pfaffian is a homogeneous polynomial of degree $k$ in $\K[x_{i,j}, \{i,j\} \in \binom{[2k]}2]$, with one monomial for each of the $(2k-1)!!$  perfect matchings in $[2k]$. 

For each $n\ge 2k+2$, let $I_k^P(n)$ be the ideal in $\K[x_{i,j}, \{i,j\} \in \binom{[n]}2]$
generated by all the Pfaffians of degree $k+1$.
Let $\Pf{k}{n}\subset\K^{\binom{n}{2}}$ be the corresponding algebraic variety. 
That is, points in $\Pf{k}{n}$ are antisymmetric $n\times n$ matrices  with coefficients in $\K$ and of rank at most $2k$. 
It is well-known and easy to see that $\Pf{1}{n}$ equals the Grassmannian $\Gr2n$ in its Pl\"ucker embedding and, as pointed out in \cite{PacStu}, 
$\Pf{k}{n}$ equals the $k$-th secant variety of it.

For $k=1$,  Pfaffians are a universal Gr\"obner basis of $I_k^P(n)$ \cite{PacStu,SpeStu}. For $k>1$ they are not (see Example 2.11 in \cite{CreSan:Pfaffians}), but they are a Gr\"obner basis for certain choices of monomial orders, such as fp-positive vectors (see Lemma 2.7 there). Together with Proposition \ref{indep}, this proves Theorem \ref{thm:hyperconnectivity}.

Let $\phi$ be the morphism $K[x_{ij}]_{(i,j)\in S_k(n)}\to K[x_{ij}]_{1\le i<j\le n}$ given by $\phi(x_{ij})=x_{i,n+1-j}$. That is, a relabelling that reverses the order of the second indices.

\begin{lemma}\label{initial}
	Let $I_k^P(n)$ the ideal generated by the Pfaffians of degree $k+1$ and $v_P$ the weight vector that assigns to each edge $(a,b)$ weight equal to $b-a$. Then, $\phi(I_k(n))=\ini_{v_P}(I_k^P(n))$.
\end{lemma}
\begin{proof}
	It is enough to see that the initial form of a Pfaffian of degree $k+1$, respect to $v_P$, coincides with the image by $\phi$ of a size $k+1$ minor contained in $S_k(n)$, and all the possible minors are obtained this way. Let $f$ be the Pfaffian of $\{i_1,\ldots,i_{2k+2}\}$, with $i_1<\ldots<i_{2k+2}$. Then, every term $f_1$ in $f$ corresponds to a perfect matching in this set: $\{\{a_1,b_1\},\ldots,\{a_{k+1},b_{k+1}\}\}$.
	\[
	v_P(f_1)=\sum_{j=1}^{k+1}(b_j-a_j)\le\sum_{j=k+2}^{2k+2}i_j-\sum_{j=1}^{k+1}i_j
	\]
	with equality if and only if $\{a_1,\ldots,a_{k+1}\}=\{i_1,\ldots,i_{k+1}\}$ and $\{b_1,\ldots,b_{k+1}\}=\{i_{k+2},\ldots,i_{2k+2}\}$. The terms that attain this maximal weight are exactly the permutations in the minor for the rows $\{i_1,\ldots,i_{k+1}\}$ and columns $\{i_{k+2},\ldots,i_{2k+2}\}$, so that minor is the initial form of $f$. As all the column indices are greater than $k+1$ and $i_{k+2}>i_{k+1}$, this minor is the image by $\phi$ of a minor contained in $S_k(n)$.
	
	On the other hand, if we start with a minor contained in $S_k(n)$, after applying $\phi$, the first column is greater than the last row, and as already shown, it is the initial form of a Pfaffian.
\end{proof}

Recall that the algebraic matroid of the Pfaffian ideal $I_k^P(n)$ coincides with the generic hyperconnectivity matroid of $n$ vectors in dimension $2k$ (see, e.g., \cite{CreSan:Pfaffians}). With this in mind, Proposition \ref{indep} and Lemma~\ref{initial}  have the following consequence, which is highly non-trivial when looking at the matrices defining hyperconnectivity.

\begin{corollary}\label{bipfree}
	If the bipartization of $E$ is free in the generic hyperconnectivity matroid in dimension $k$, then $E$ is free in the generic hyperconnectivity matroid in dimension $2k$.
\end{corollary}
\begin{proof}
	Let $E$ be dependent in dimension $2k$ and $E_1$ the bipartization of $E$. As we know, $E$ is algebraically dependent in $I_k^P(n)$, so there is a polynomial $f$ in $I_k^P(n)$ using only variables in $E$. Then $\ini_{v_P}(f)=\phi(g)$ where $g$ is a polynomial in $I_k(n)$. If a variable $x_{ij}$ is used in $g$, $\phi(g)$ uses $x_{i,n+1-j}$, so $f$ also uses it and $\{i,n+1-j\}\in E$. By the definition of bipartization, $(i,j)\in E_1$. Hence, all the variables in $g$ are in $E_1$, and $E_1$ is algebraically dependent in $I_k(n)$, that is, it is dependent in dimension $k$.
\end{proof}

However, the converse is not true: $K_{2k+2}\setminus\{1,2\}$ is a basis in dimension $2k$ and its reduced bipartization is $K_{k+1,k+1}$, which is a circuit in dimension $k$.

In what follows, we will call $\Grob_k(S)$ the Gr\"obner cone of the $(k+1)$-crossing terms of $I_k(S)$ (that is, the cone of the vectors for which the $(k+1)$-crossing attains the maximum weight in each $A$ and $B$ with size $k+1$), and $\Grob_k(n)$ analogously. As a slight abuse of notation, we will call $\phi$ the linear morphism $\R^{S_k(n)}\to\R^\bnn$ that permutes the coordinates in the same way as $\phi$ (and fills with an arbitrary value, such as 0, the rest of coordinates in $\bnn$).

\begin{corollary}\label{grobcone}
	$v\in\Grob_k(n)\Leftrightarrow v_P+\epsilon \phi(v)\in\Grob_k^P(n)$ for $\epsilon$ small enough. Also, for weights in the interior of the cone, the $(k+1)$-size minors are a Gr\"obner basis. 
\end{corollary}
\begin{proof}
	For the first part,
	\begin{align*}
		v\in\Grob_k(n) & \Leftrightarrow \ini_v(I_k(n)) \text{ is the ideal of (bipartite) $(k+1)$-crossings} \\
		& \Leftrightarrow \phi(\ini_v(I_k(n))) \text{ is the ideal of (circular) $(k+1)$-crossings} \\
		& \Leftrightarrow \ini_{\phi(v)}(\phi(I_k(n))) \text{ is the ideal of $(k+1)$-crossings} \\
		& \Leftrightarrow \ini_{\phi(v)}(\ini_{v_P}(I_k^P(n))) \text{ is the ideal of $(k+1)$-crossings} \\
		& \Leftrightarrow \ini_{v_P+\epsilon \phi(v)}(I_k^P(n)) \text{ is the ideal of $(k+1)$-crossings} \\
		& \Leftrightarrow v_P+\epsilon \phi(v)\in\Grob_k^P(n)
	\end{align*}
For the second, let $f\in I_k(n)$ and $v$ an interior vector of the cone. We need to see that the leading term $f_0$ of $f$ (that is a single monomial because $v$ is interior) is multiple of a $(k+1)$-crossing monomial. As $\phi(I_k(n))$ is the initial ideal of $I_k^P(n)$ with weight vector $v_P$, $\phi(f)=\ini_{v_P}(g)$ for a $g\in I_k^P(n)$. That is, 
\[\phi(f_0)=\ini_{\phi(v)}(\phi(f))=\ini_{\phi(v)}(\ini_{v_P}(g))=\ini_{v_P+\epsilon \phi(v)}(g)\]
for $\epsilon$ small enough. Now we know that $v_P+\epsilon \phi(v)\in\Grob_k^P(n)$ and by Theorem 2.7 in \cite{CreSan:Pfaffians}, Pfaffians are a Gr\"obner basis with respect to $v_P+\epsilon \phi(v)$, so $\phi(f_0)$ is multiple of a circular $(k+1)$-crossing monomial, and $f_0$ is multiple of a bipartite one.
\end{proof}

Now the arguments will follow a similar path than in the case of Pfaffians: fp-positive weight vectors are in the Gr\"obner cone.

\begin{definition}
	For a Ferrers diagram $S$ of $[n_1]\times[n_2]$, let $\overline{S}$ be the Ferrers diagram of $[n_1-1]\times[n_2-1]$ obtained by deleting the elements of the form $(1,j)$ or $(i,1)$ and decrementing the rest of indices.
	
	Let $\delta:\R^S\to\R^{\overline{S}}$ given by
	\[\delta(v)_{ij}=v_{i,j+1}+v_{i+1,j}-v_{ij}-v_{i+1,j+1}\]
\end{definition}

\begin{proposition}\label{coor}
	$\delta(v)$ determines $v$ except for a constant term added to each row or column:
	\[v_{ij}=v_{i1}+v_{1j}-v_{11}-\sum_{r=1}^{i-1}\sum_{s=1}^{j-1}\delta(v)_{rs}\]
\end{proposition}

\begin{theorem}\label{fp1}
	Let $v\in\R^S$ be a weight vector for the variables of $I_k(S)$. The following are equivalent:
	\begin{enumerate}
		\item All coordinates of $\delta(v)$ are nonnegative.
		\item For $(i,j),(i',j')\in S$, $i<i'$ and $j<j'$, $v_{ij'}+v_{i'j}\ge v_{ij}+v_{i'j'}$.
		\item For any subsets $A\subset[n_1]$ and $B\subset[n_2]$ with the same size such that $A\times B\subset S$, the weights given by $v$ to perfect matchings between them are monotone with respect to swaps that increase crossings.
		\item For any subsets $A\subset[n_1]$ and $B\subset[n_2]$ with the same size such that $A\times B\subset S$, the maximum weight given by $v$ to perfect matchings between them is attained at the crossing.
	\end{enumerate}
A vector $v$ satisfying any of these conditions will be called four-point positive or fp-positive.
\end{theorem}
\begin{proof}
	Obviously we have $4\Rightarrow 3\Leftrightarrow 2\Rightarrow 1$: the part 2 is a particular case of 4 when $|A|=|B|=2$, and it is equivalent to 3, and in turn, part 1 is a particular case of 2.
	
	Suppose that part 1 holds. Then, by Proposition \ref{coor} or directly from the definition,
	\[v_{ij'}+v_{i'j}-v_{ij}-v_{i'j'}=\sum_{r=i}^{i'-1}\sum_{s=1}^{j'-1}\delta(v)_{rs}\ge 0\]
	and we have part 2.
	
	Suppose now that part 3 holds. Then we can go from any matching between $A$ and $B$ to the complete crossing by swaps that create crossings, and by part 3 this process increases the weight in each step, so the complete crossing will have higher weight and part 4 holds.
\end{proof}

\begin{remark}
	This definition of fp-positive vectors is similar, but not exactly equal, to the one found in \cite{CreSan:Pfaffians}. In that article, we define a coordinate $w$ for each edge $\{i,j\}$, and then show that the ones corresponding to sides of the polygon are in the lineality space of the four-point positive cone, that is, there are $\binom{n}{2}-n$ inequalities defining four-point positive vectors, that have dimension $\binom{n}{2}$.
	
	For $\delta(v)$ as defined here, it has no coordinates in the lineality space by Theorem \ref{fp1}, and its dimension is $n_1+n_2-1$ less than that of $v$, that is exactly the dimension of the lineality space (given as constants added to each row or column). In this paper, for the standard case, the fp-positive vectors have dimension $\binom{n}{2}-k(k+1)$ and are defined by $\binom{n-2}{2}-k(k-1)$ inequalities. The lineality space is larger: it has dimension $2(n-k)-3$. That is, the projection of that cone is strictly contained in this cone.
\end{remark}

Let $\Delta:\R^\bnn\to\R^\bnn$ be the morphism that sends the $v$ coordinates to the $w$ coordinates as defined for the circular crossings (except by a factor 2 that appears there):
\[\Delta(v)_{ij}=v_{ij}+v_{i+1,j+1}-v_{i,j+1}-v_{i+1,j}\]
(Note that $i$ and $j$ are here taken modulo $n$, and $v_{ii}=0$.)

We can easily see that $\Delta(\phi(v))=\phi(\delta(v))$ when restricted to the coordinates in $\overline{S}$, that is, the operation $\delta$ is the bipartite correspondence of the circular operation $\Delta$.

In the standard case where $S=S_k(n)$, we can characterize exactly which vectors are in the Gr\"obner cone. Let $e_{ij}$, for $(i,j)\in S$, be a vector in the standard basis, and $f_{ij}$, for $(i,j)\in \overline{S}$, the vector whose image by $\delta$ is $e_{ij}$.
\begin{theorem}\label{thm:grobcone}
	$\Grob_k(n)$ is given by the following inequalities:
	\begin{align}
		\sum_{r\ge i,s\ge j, s-r\le k+1}\delta(v)_{rs} \ge 0 &\quad \text{ if } 2\le i+j\le k \\
		\delta(v)_{ij}\ge 0 &\quad \text{ if } k+1\le i+j\le n-k-1 \\
		\sum_{r\le i,s\le j, s+r\ge n}\delta(v)_{rs} \ge 0 &\quad \text{ if } n-k\le i+j\le n-2
	\end{align}
	Its rays are:
	\begin{itemize}
		\item $-e_{ij}$ with $2\le i+j\le k+1$
		\item $f_{ij}$ with $k+2\le i+j\le n-k-2$
		\item $-e_{i+1,j+1}$ with $n-k-1\le i+j\le n-2$
	\end{itemize}
\end{theorem}

\begin{proof}
	The cone we are looking for is the ``link'' of the smallest face containing $v_P$ in $\Grob_k^P(n)$. There is an expression for $\Grob_k^P(n)$ in \cite[Theorem 2.9]{CreSan:Pfaffians}: its inequalities are
	\begin{align}
		\Delta(v)_{ij}&\ge 0\quad \text{if } |j-i|>k  & \text{(long inequalities)}\label{eqrel} \\
		\sum_{i'\le i<j\le j'\le i'+k+1} \Delta(v)_{i'j'} & \ge 0\quad \text{if } 2\le |j-i|\le k & \text{(short inequalities)} \label{eqirrel}
	\end{align}
	Using the formula to compute $\Delta(v_P)$, we get that $\Delta(v_P)_{ij}=0$ except if $j=n$, in which case $\Delta(v_P)_{in}=n-i+(i+1)-1-n+(i+1)-i+1=2$. That is, taking the link of $v_P$ is equivalent to setting the coordinates $\{i,n\}$ to a value much bigger than the rest, and all inequalities including a $\Delta(v)_{in}$ become irrelevant. Replacing $v$ by $\phi(v)$ and using that $\Delta(\phi(v))=\phi(\delta(v))$, we have the desired result.
	
	The rays are deduced by setting all the inequalities to hold with equality except one. If we do this with a coordinate in the range $k+2\le i+j\le n-k-2$, we get just $\delta(v)_{ij}>0$ and the rest equal to 0, which is $f_{ij}$. If the coordinate is outside this range, the result is
	\[\delta(v)_{ij}=-\delta(v)_{i,j-1}=-\delta(v)_{i-1,j}=\delta(v)_{i-1,j-1}>0\]
	and the rest 0, which leads exactly to $-e_{ij}$, or
	\[\delta(v)_{ij}=-\delta(v)_{i,j+1}=-\delta(v)_{i+1,j}=\delta(v)_{i+1,j+1}>0\]
	which leads to $-e_{i+1,j+1}$.
\end{proof}

\begin{theorem} \label{fp}
	For any $k$-full set $S$, $\Grob_k(S)$ contains the four-point positive cone. For weights in the interior of the cone, the $(k+1)$-size minors are a Gr\"obner basis. 
\end{theorem}

\begin{proof}
	The first sentence is a consequence of the point 4 of Theorem \ref{fp1}.
	
	For the second one, Corollary \ref{grobcone} already proves it for $S=S_k(n)$. Now we will prove it for any other $k$-full $S$. In first place, it is easy to see that $S$ can be obtained from $S_k(n)$ for some $n$ by deleting rows and columns. In terms of $I_k(n)$, this is equivalent to restrict the ideal to a subset of variables.
	
	Now, we can extend the weight vector $v$ from $S$ to $S_k(n)$ interpolating linearly the values in the new rows and columns. The effect of this in $\delta(v)$ is to divide the value in a coordinate between several ones. This preserves positivity of the coordinates, so we are still in the Gr\"obner cone. Applying Corollary \ref{grobcone} to this case, we have that the minors are a Gr\"obner basis, so the restriction of the ideal also is.
\end{proof}

Putting together Lemma \ref{indep} and Theorem \ref{fp}, we get:

\begin{corollary}
	$(k+1)$-free bipartite graphs defined in a set $S$ are independent in the algebraic matroid of $\M_k(S)$. Hence, the maximal of these graphs (that have $k(n_1+n_2)-k^2$ edges) are bases.
\end{corollary}

This has two consequences. On one hand, together with Theorem \ref{thm:hyper-algebraic}, it implies Theorem \ref{thm:bipbases}. On the other hand, we can apply to this case Theorem \ref{thm:algebraic}, resulting in:
\begin{theorem}
	\label{thm:matroid}
	Let $T\subset [n_1]\times[n_2]$.
	\begin{enumerate}
		\item If $T$  is $(k+1)$-free and $\K$ is algebraically closed, then for any generic choice of values $v\in \K^T$ there is at least one matrix in $\K^{n_1\times n_2}$ of rank $\le k$ with the entries prescribed by $v$.
		\item If $T$ contains a maximal $(k+1)$-free graph then for any choice of values $v\in \K^T$ there is only a finite number (maybe zero) of matrices in $\K^{n_1\times n_2}$ of rank $\le k$ with those prescribed entries.
	\end{enumerate}
\end{theorem}

\subsection{Tropical geometry}\label{sec:tropical}
In \cite{CreSan:Pfaffians}, we prove that the tropical prevariety of antisymmetric matrices with size $n$ and rank at most $2k$ contains a fan isomorphic to the $k$-associahedron in $n$ vertices. It turns out that the tropical prevariety of matrices with size $n-k-1$ and rank at most $k$ (projected to a subset of the coordinates) also contains a fan isomorphic to the $(k-1)$-associahedron in $n-2$ vertices.

More concretely, these are the results about Pfaffians, where $\Pf{k}{n}$ is the Pfaffian variety, $\PV{k}{n}$ the tropical prevariety, $\Grob_k^P(n)$ the Gr\"obner cone for $(k+1)$-crossings and $\PVplus{k}{n}=\PV{k}{n}\cap\Grob_k^P(n)$:
\begin{theorem}[\cite{CreSan:Pfaffians}, Theorem 4.3]
	Let $v$ be a vector in $\Grob_k^P(n)$. Then, $v\in\PVplus{k}{n}$ if and only if the support of $\Delta(v)$ is $(k + 1)$-free.
\end{theorem}
\begin{theorem}[\cite{CreSan:Pfaffians}, Corollary 4.5]\label{freevarpf}
	$\PVplus{k}{n}\subset\trop(\Pf{k}{n})$. Moreover, $\PVplus{k}{n}\subset\trop^+(\Pf{k}{n})$.
\end{theorem}

Then, for the determinantal prevariety, we have analogously defined $\M_k(S)$ as the projection of the determinantal variety into a subset $S$ of coordinates, and its particular case $\M_k(n)$. So we can define $M_k(n)$ as the tropical prevariety of $I_k(n)$, $\Grob_k(n)$ the Gr\"obner cone for the $(k+1)$-crossings, and $\Mplus{k}{n}=M_k(n)\cap\Grob_k(n)$. Then, the following is true:
\begin{theorem}\label{freeprev}
	Let $v$ be a vector in $\Grob_k(n)$. Then, $v\in\Mplus{k}{n}$ if and only if the support of $\delta(v)$ is $k$-free.
\end{theorem}

\begin{proof}
	As we already know, $v\in\Grob_k(n)$ is equivalent to $v_P+\epsilon \phi(v)\in\Grob_k^P(n)$ for small $\epsilon$.
	\begin{align*}
		v\in\Mplus{k}{n} & \Leftrightarrow \ini_v(f) \text{ is not a monomial for any $(k+1)$-size minor } f \\
		& \Leftrightarrow \ini_{v_P+\epsilon \phi(v)}(g) \text{ is not a monomial for any Pfaffian } g \\
		& \Leftrightarrow v_P+\epsilon \phi(v)\in\PVplus{k}{n} \\
		& \Leftrightarrow \Supp \Delta(v_P+\epsilon \phi(v))\text{ is $(k+1)$-free} \\
		& \Leftrightarrow \Supp (\Delta(v_P)+\epsilon \phi(\delta(v)))\text{ is $(k+1)$-free}
	\end{align*}
	Now we only need to show that $\Delta(v_P)+\epsilon \phi(w)$ has $(k+1)$-free support if and only if $w$ has $k$-free support, for $w\in\R^{\overline{S_k(n)}}$.
	
	We already know that $\Delta(v_P)_{ij}=0$ except if $j=n$, in which case $\Delta(v_P)_{in}=2$ (see proof of Theorem \ref{thm:grobcone}).
	
	If there is a (circular) $(k+1)$-crossing in $\Delta(v_P)+\epsilon \phi(w)$, only one coordinate can have $j=n$, and the rest are zero in $\Delta(v_P)$ and consequently nonzero in $\phi(w)$, so we have a $k$-crossing in $\phi(w)$ not including $n$, which translates to a (bipartite) $k$-crossing in $w$.
	
	Conversely, if there is a $k$-crossing in $w$, it translates to a $k$-crossing in $\phi(w)$ that does not include $n$. Let this be $\{a_1,b_1\},\{a_2,b_2\},\ldots,\{a_k,b_k\}$ with $a_1<\ldots<a_k$, $b_1<\ldots<b_k<n$ and $b_1-a_k\ge 2$. This implies that $\{a_1,b_1\},\{a_2,b_2\},\ldots,\{a_k,b_k\},\{a_k+1,n\}$ is a $(k+1)$-crossing in $\Delta(v_P)+\epsilon \phi(w)$.
\end{proof}

\begin{theorem}\label{freevar}
	$\Mplus{k}{n}\subset\trop(\M_k(n))$. Moreover, $\Mplus{k}{n}\subset\trop^+(\M_k(n))$.
\end{theorem}

\begin{proof}
	If $v\in\Mplus{k}{n}$, $v_P+\epsilon \phi(v)\in\PVplus{k}{n}$ for some $\epsilon>0$. By Theorem \ref{freevarpf}, $v_P+\epsilon \phi(v)\in\trop(\Pf{k}{n})$, so $\ini_{v_P+\epsilon \phi(v)}(I_k^P(n))$ contains no monomials. But $\ini_{v_P+\epsilon \phi(v)}(I_k^P(n))=\ini_{\phi(v)}(\ini_{v_P}(I_k^P(n)))=\phi(\ini_v(I_k(n)))$, so $v\in\trop(\M_k(n))$. The proof of the second part is analogous.
\end{proof}

Theorems \ref{freeprev} and \ref{freevar} are also true if we replace $S_k(n)$ by any $k$-full $S$, because all them can be obtained from $S_k(n)$ for some $n$.

To conclude this section, we see that projecting the $\Mplus{k}{n}$ to a complete fan, in order to realize the multiassociahedron, is equivalent to projecting a link of $\PVplus{k}{n}$ into a complete fan. This confirms Lemma \ref{lemma:monotone} that the $(k-1)$-associahedron in $n-2$ vertices is a link of the $k$-associahedron in $n$ vertices.

\section{Realizability in the moment curve}\label{sec:real}

In this section we turn our attention to realizability of $\OvAss{k}{2k+2}$ as a fan. 

\subsection{The case $n\le 2k+3$}

The case $n=2k+2$ is easy with what we know:

\begin{theorem}
	For $n=2k+2$, every configuration in the moment curve realizes $\OvAss{k}{2k+2}$ as a polytopal fan.
\end{theorem}
\begin{proof}
	In this case we have only one flip graph, which is the complete graph. Bipartizing it, we get the complete bipartite graph $K_{k+1,k+1}$. The flipping edges are the diameters of the original graph, that become $(1,k+1),(2,k),\ldots,(k+1,1)$ after bipartizing. Corollary \ref{cor:signKd+1} predicts all the signs of the circuit: the flipping edges have the same sign, so the condition ICoP is satisfied.
	
	Respect to the condition about elementary cycles, there is nothing to prove here because all them have length 3. So the fan is realized in any position, and it is automatically polytopal because it is the fan of a simplex.
\end{proof}

Now we will find necessary and sufficient conditions so that the complete fan is realized for $n=2k+3$. First, a definition:

\begin{definition}
	Let $S$ be a set of $k+2$ points in the projective space $\R\proj^{k-1}$ in general position and let $a$, $b$, $c$ and $d$ be four of them. Let $T=S\setminus\{a,b,c,d\}$.
	The \emph{relative cross-ratio} of $a,b,c,d\in S$, denoted as $(a,b;c,d)_S$, is the cross-ratio of the four points obtained by conical projection of  $a,b,c,d$ with center at $T$ onto a projective line. Equivalently, it equals the cross-ratio of the four planes spanned by $Ta$, $Tb$, $Tc$, $Td$.
	That is,
	\[(a,b;c,d)_S=\frac{|acT|\cdot|bdT|}{|adT|\cdot|bcT|}\]
	where $|xyT|$ denotes the $k\times k$ determinant whose rows are $x$, $y$ and the points in $T$.
\end{definition}

In the case of points $\{t_i: i\in [n]\}$ in the moment curve, this formula simplifies to
\[(a,b;c,d)_\bt=\frac{(t_c-t_a)(t_d-t_b)}{(t_d-t_a)(t_c-t_b)}\]

Note that the bipartized graph for $n=2k+3$ has $k+2$ vertices in each side: only the central vertex $k+2$ is common to both sides (as the last in each side), each vertex $i<k+2$ becomes $i$ in the left side, and each $j>k+2$ becomes $2k+4-j$ in the right side.

\begin{definition}
	Given three relevant edges $\{i_l,j_l\}, l=1,2,3$ of the $(2k+3)$-gon, where $i_1<i_2<i_3<k+2<j_1<j_2<j_3$, and a position $\bt$ in the moment curve, we will say that the three edges are \textit{correctly located} if $(i_1,i_2;i_3,k+2)_\bt>(j_1',j_2';j_3',(k+2)')_\bt$, where $j_l'=2k+4-j_l$ is the new index of vertex $j_l$.
\end{definition}

Note that, despite its asymmetric appearance, being correctly located is a symmetric relation: if we reverse the order of the vertices, the relation becomes $(j_3',j_2';j_1',(k+2)')_\bt>(i_3,i_2,i_1,k+2)_\bt$, which is equivalent to the previous one.

\begin{lemma}\label{lemma:correct}
	Given four relevant edges $e_l=\{i_l,j_l\}, l=1,2,3,4$ of the $(2k+3)$-gon, where $i_1<i_2<i_3<i_4<k+2<j_1<j_2<j_3<j_4$, if $\{e_1,e_2,e_3\}$ and $\{e_2,e_3,e_4\}$ are correctly located, $\{e_1,e_2,e_4\}$ is also correctly located (and, by symmetry, $\{e_1,e_3,e_4\}$, but we do not need that).
\end{lemma}
\begin{proof}
	By a projective transformation of the left vertices and another for the right ones, we can make $t_{i_1}=t_{j_1}'=0$, $t_{i_2}=t_{j_2}'=1$ and $t_{k+2}=t_{k+2}'=\infty$. Denote $a_l=t_{i_l}$ and $b_l=t_{j_l}'$, for $l=3,4$. We must have $a_4>a_3>1$ and $b_4>b_3>1$.
	
	By hypothesis $(i_1,i_2;i_3,k+2)_\bt>(j_1',j_2';j_3',(k+2)')_\bt$, that is, $a_3/(a_3-1)>b_3/(b_3-1)$, which, as $a_3$ and $b_3$ are greater than 1, implies $a_3<b_3$.
	
	Also by hypothesis, $(i_2,i_3;i_4,k+2)_\bt>(j_2',j_3';j_4',(k+2)')_\bt$, that is,
	\[\frac{a_4-1}{a_4-a_3}>\frac{b_4-1}{b_4-b_3}\]
	\[b_3-b_4-b_3a_4>a_3-a_4-a_3b_4\]
	\[b_4+b_3a_4-a_4-a_3b_4<b_3-a_3<a_4(b_3-a_3)\]
	\[b_4-a_4-a_3b_4+a_3a_4<0\]
	\[(b_4-a_4)(1-a_3)<0\]
	As $a_3>1$, this implies $b_4>a_4$ and
	\[(i_1,i_2;i_4,k+2)_\bt=\frac{a_4}{a_4-1}>\frac{b_4}{b_4-1}=(j_1',j_2';j_4',(k+2)')_\bt\]
	as we wanted.
\end{proof}

\begin{lemma}
	Let $1\le i_1<i_2<k+2<j_1<j_2\le 2k+3$ be such that the bipartization of $C:=K_{2k+3}\setminus\{\{i_1,j_1\}, \{i_2,j_2\}\}$ is a circuit. Consider a configuration $\bt$ in the moment curve for the vertices of this bipartization and let $\lambda$ be the unique  (modulo a scalar factor) dependence for $C$ in the bipartite polynomial rigidity matrix $P_k(\bt)$. Let $\{i_3,j_3\}$ be the first edge in the odd path of $S\setminus \{\{i_1,j_1\}, \{i_2,j_2\}\}$ (that is, $\{i_1\pm 1,j_1\}$ or $\{i_1,j_1\pm 1\}$), and let $\{i_0,j_0\}$ be another edge in the same path at a distance $d$ from the first. Then, we have that
	\[
	\sign(\lambda_{i_0j_0})=(-1)^d\sign(\lambda_{i_3j_3})
	\]
	if and only if $\{e_0,e_1,e_2\}$ are correctly located.
\end{lemma}

\begin{proof}
	Suppose first that $i_0<i_1<i_2<j_0<j_1=i_1+k+1<j_2=i_2+k+2$. Then, by the definition of star order, $i_3=i_1-1$, $j_3=j_1$, $i_0+j_0= i_1+j_1-d-1$.
	
	Bipartizing the circuit, we get $K_{k+2,k+2}$ minus three edges. In the resulting graph, the degree of $j_1$ is $k+1$ and by Lemma \ref{lemma:chsign},
	\[\sign(\lambda_{i_1-1,j_1})=(-1)^{i_1-i_0-1}\sign(\lambda_{i_0j_1})\]
	so the condition to be checked reduces to
	\[\sign(\lambda_{i_0j_0}\lambda_{i_0j_1})=(-1)^{d+1+i_1-i_0}=(-1)^{j_1-j_0}\]
	
	The circuit is obtained by a repeated bipartite coning from the graph of $K_{4,4}$ minus three edges, so that the original eight vertices become $\{i_0,i_1,i_2,k+2\}$ in the left side and $\{j_2',j_1',j_0',(k+2)'\}$ in the right side. That is, the three missing edges are $22'$, $31'$ and $44'$. By Corollary \ref{cor:signconing}, the condition reduces to $\sign(\lambda_{12}\lambda_{13})=-1$, because this sign is inverted $i_3-j_2-1$ times (once for each vertex added between $j_2$ and $i_3$).
	
	Swapping (via a projective transformation of the line) the positions of vertices $1'$ and $3'$ will not alter the signs in the first three vertices, because this is a linear transformation in dimension 2 followed by a rescaling that has the same sign in these three vertices (Proposition \ref{prop:invariance}). After the change, we are in the situation of Theorem \ref{cube}, and the condition is now $\sign(\lambda_{11}\lambda_{12})=-1$. This happens if and only if $(1,2;3,4)>(1',2';3',4')$. As contracting does not change the relative cross-ratios, this means $(i_1,i_2;i_3,k+2)_\bt>(j_1',j_2';j_3',(k+2)')_\bt$, as we want.
	
	The remaining cases lead, analogously, to $(i_2,i_3;k+2,i_1)_\bt<(j_2',j_3';(k+2)',j_1')_\bt$ and $(i_3,k+2;i_1,i_2)_\bt>(j_3',(k+2)';j_1',j_2')_\bt$. The third condition is the same as the first one, and the second one is also equivalent, because of the order of the vertices.
\end{proof}

We are now ready to prove the main result for $n=2k+3$. Remember that an \textit{octahedral} triangulation is one whose three missing edges are disjoint.

\begin{theorem}\label{thm:star}
	Let $\bt=(t_1,t_2 ,\dots ,t_{k+2},t_1',t_2',\dots,t_{k+2}')$ be a configuration in the moment curve in $\R^k$. The following are equivalent:
	\begin{enumerate}
		\item	$P_k(\bt)$ realizes $\OvAss{k}{2k+3}$ as a complete fan.
		\item For every octahedral $k$-triangulation $T$ the three edges not in $T$ are correctly located.
	\end{enumerate}
\end{theorem}
\begin{proof}
	If (1) holds, the condition ICoP is satisfied in the flips from an octahedral triangulation. In particular, with the notations of the previous Lemma, the flip from $K_{2k+3}\setminus\{\{i_0,j_0\},\{i_1,j_1\},\{i_2,j_2\}\}$ that removes $\{i_3,j_3\}$ and inserts $\{i_0,j_0\}$, with $d$ even. By the Lemma, this is equivalent to saying that the three edges are correctly located.
	
	Now suppose that (2) holds. Then, there are two types of flips: those whose two missing edges share a vertex and the rest. For the first case, the bipartization contains a $K_{k+1,k+1}$ and we already know that the signs of the flipped edges are correct. For the second, the two missing edges are disjoint. If one of them contains the vertex $k+2$, the bipartization contains again a $K_{k+1,k+1}$. Otherwise, we are in the conditions of the Lemma. Applying it to all edges with $d$ even, that are exactly the edges that can be removed from the graph to get a triangulation, all of them need to have the same sign, hence the condition ICoP also holds here.
	
	It only remains to see the condition about elementary cycles. Let $\{e_1,e_2,e_3,e_4,e_5\}$ an elementary cycle of length 5, which means that all the paths left between the edges are even. If two edges share a vertex, for example $e_1$ and $e_2$, the flip $K_{2k+3}\setminus\{e_1,e_2\}$ contains a $K_{k+1,k+1}$, and the signs of the other three edges are automatically correct. The same happens if $k+2$ is a vertex of any edge.
	
	In the remaining cases, $\{e_1,e_2,e_3\}$ and $\{e_2,e_3,e_4\}$ are correctly located by hypothesis. By Lemma \ref{lemma:correct}, $\{e_1,e_2,e_4\}$ is also correctly located. By the Lemma, in the flip $K_{2k+3}\setminus\{e_1,e_2\}$, $e_4$ has opposite sign to $e_3$ and $e_5$, and we are done.
\end{proof}

\begin{corollary}\label{cor:2k+3}
	\begin{itemize}
		\item For $k=2$ and $n=7$, any choice of points $(t_1,t_2,t_3,t_4,t_1',t_2',t_3',t_4')$ realizes $\OvAss{2}{7}$ as a fan.
		\item For any $k$, there is a choice of points that realizes $\OvAss{k}{2k+3}$ as a fan.
	\end{itemize}
	\begin{proof}
		For $k=2$, the stated condition trivially holds in any position because there are no octahedral triangulations.
		
		For the second part, take $t_{k+2}$ and $t_{k+2}'$ very big, $t_{k+1}=t_{k+1}'=1$, $t_k=t_k'=0$ and $t_i$ lexicographically smaller than $t_{i+1}$. Then, for any octahedral triangulation, in the corresponding inequality the left hand side $(i_1,i_2;i_3,k+2)_\bt$ is large and the right hand side $(j_1',j_2';j_3',(k+2)')_\bt$ is close to 1 (because $j_1'$ and $j_2'$ are relatively close).
	\end{proof}
\end{corollary}

For $k\ge 3$, we can now prove Theorem \ref{thm:2k+6}, because any subset with $2k+3$ vertices should also realize the fan.

\begin{proof}[Proof of Theorem \ref{thm:2k+6}]
	First suppose that $k=3$ and $n=12$. This will account for all cases with $n\ge 12$, because all of them contain $\OvAss{3}{12}$.
	
	Consider the subconfiguration formed by the $9$ vertices ${1,2,3,5,6,7,8,9,10}$, where $6$ is the central vertex and $\{1,7\}$, $\{3,8\}$ and $\{5,10\}$ are three missing edges in an octahedral $3$-triangulation.
	By Theorem \ref{thm:star}, the three edges are correctly located, that is:
	\[(1,3;5,6)>(6',5';3',7')\]
	In terms of the parameters,
	\[\frac{(t_5-t_1)(t_6-t_3)}{(t_6-t_1)(t_5-t_3)}>\frac{(t_7'-t_5')(t_6'-t_3')}{(t_7'-t_6')(t_5'-t_3')}\]
	Using that the parameters are in increasing order,
	\[\frac{t_6-t_3}{t_5-t_3}>\frac{t_6'-t_3'}{t_5'-t_3'}\]
	Now we repeat the argument with the subconfiguration symmetric to the previous one, so the $t$'s get swapped with the $t'$'s, and we get
	\[\frac{t_6'-t_3'}{t_5'-t_3'}>\frac{t_6-t_3}{t_5-t_3}\]
	This is a contradiction.
	
	Now we will get a similar contradiction for $k\ge 4$ and $n\ge 2k+4$. As before, we can take $n=2k+4$.
	
	Consider the subconfiguration formed by the first $2k+3$ vertices. The complete graph in those vertices minus the three edges $\{2,k+3\}$, $\{3,k+5\}$ and $\{k+1,2k+2\}$ is an octahedral $k$-triangulation (note that, as $k\ge 4$, $2k+2>k+5$).
	
	Same as before, the three edges are correctly located and
	\[(2,3;k+1,k+2)>((k+2)',k';3',(k+3)')\]
	\[\frac{(t_{k+1}-t_2)(t_{k+2}-t_3)}{(t_{k+2}-t_2)(t_{k+1}-t_3)}>\frac{(t_{k+3}'-t_k')(t_{k+2}'-t_3')}{(t_{k+3}'-t_{k+2}')(t_k'-t_3')}\]
	This implies
	\[\frac{t_{k+2}-t_3}{t_{k+1}-t_3}>\frac{t_{k+2}'-t_3'}{t_k'-t_3'}>\frac{t_{k+2}'-t_3'}{t_{k+1}'-t_3'}\]
	and we are done by applying symmetry.
\end{proof}

\subsection{Experimental results}
In this section we report on some experimental results. In all of them we choose real parameters $\bt= \{t_1< t_2<\dots<t_{n-k-1},t_1'<t_2'<\ldots<t_{n-k-1}'\}$ and computationally check whether the configuration of rows of $P_k(\bt)$ realizes $\OvAss{k}{n}$ first as a collection of bases, then as a complete fan, and finally as the normal fan of a polytope. Note that, after the results in the previous section, the only fans that we may expect to realize in this way are those with $k=2$ and those with $k=3$ and $n\le 11$.

This is similar to \cite[section 4.3]{CreSan:Multi}. For the experiments we have written python code which, with input $k$, $n$ and the parameters $\bt$, first constructs the set of all $k$-triangulations and then checks the three levels of realizability as follows:
\begin{enumerate}
	\item Realizability as a collection of bases amounts to computing the rank of the submatrix $P_k(\bt)|_T$ corresponding to each $k$-triangulation $T$.
	\item For realizability as a fan we first check the ICoP  property, which amounts to computing the signs of certain dependences among rows  of $P_k(\bt)$. Then, for a particular $k$-triangulation (the so-called \emph{greedy} one) we check that the sum of its rays is not contained in the cone of any other $k$-triangulation. This property, once we have ICoP, is equivalent to being realized as a fan, by parts (2) and (3) of \cite[Theorem~2.11]{CreSan:Multi}.%
	
	The \emph{greedy} $k$-triangulations is the (unique) one containing all the irrelevant edges and the complete bipartite graph $[1,k]\times [k+1,n]$, and only those. It is obvious that these edges do not contain any $(k+1)$-crossing and we leave it to the reader to verify that the number of relevant ones is indeed $k(n-2k-1)$.
	
	\item For realizability as a polytope we then need to check feasibility of a linear system of inequalities:
	\begin{lemma}[{\cite[Lemma 2.14]{CreSan:Multi}}]
		\label{lemma:inequalities}
		Let $\Delta$ be a simplicial complex with vertex set $V$ and dimension $D-1$, and assume it is a 
		triangulation of a totally cyclic vector configuration $\VV\subset \R^D$. Then,
		a lifting vector $(f_{i})_{i\in V}$ is valid if and only if for every facet $\sigma\in \Delta$ and element $a\in V\setminus \sigma$ the inequality
		\[
			\sum_{i\in C}\omega_{i}(C)f_{i}>0
		\]
		holds, where $C=\sigma \cup\{a\}$ and $\omega(C)$ is the vector of coefficients in the unique (modulo multiplication by a positive scalar) linear dependence in $\VV$ with support in $C$, with signs chosen so that $\omega_{a}(C)>0$ for the extra element.
	\end{lemma}
	
	Here, without loss of generality we can assume that the lifting vector $f_{ij}$ is zero in all edges of a particular $k$-triangulation, and we again use the greedy one. This reduces the number of variables in the feasibility problem from $n(n-1)/2-k(k+1)$ to $(n-2k)(n-2k-1)/2$, a very significant reduction for the values of $(n,k)$ where we can computationally construct $\OvAss{k}{n}$. Apart of the computational advantage, it saves space when displaying a feasible solution; in all the tables in this section we show only the non-zero values of $f_{ij}$, which are those of relevant edges  contained in $[k+1,n]$. We do this in the tables in this section.
	Note that, as a by-product, taking all the $f_{ij}$'s of a particular $k$-triangu\-lation equal to zero makes the rest strictly positive.
\end{enumerate}

\begin{remark}
	\label{rem:monotone}
	If a choice of parameters realizes $\OvAss{k}{n}$ (at any of the three levels) for a certain pair $(k,n)$ then deleting any $j$ of the parameters (in each side) the same choice realizes $\OvAss{k'}{n-j}$ for any $k'$ with $ k-j/2\le k'\le k$. This follows from Lemma~\ref{lemma:monotone} plus the fact that each of the three levels of realization is preserved by taking links.
\end{remark}

Our first experiment is taking equispaced parameters. Since an affine transformation in the space of parameters produces a linear transformation in the rows of $P_k(\bt)$, we take without loss of generality  $\bt=(1,2,3,\dots,n)$. We call these the \emph{standard positions} along the parabola.

For $k\ge 3$ and $n\ge 2k+3$ we show in Theorem \ref{thm:Desargues} that standard positions do not even realize $\OvAss{k}{n}$ as a collection of bases. Hence, we only look at $k=2$.

\begin{lemma}
	\label{lemma:standard}
	Let $\bt=\{1,2,\dots,n\}$ be standard positions for the parameters. Then:
	\begin{enumerate}
		\item Standard positions for $P_{2}(\bt)$ realize $\OvAss{2}{n}$ as the normal fan of a polytope if and only if $n\le 8$.
		
		\item The near-lexicographic positions  $t_i=2^{(i-1)^2}$ for $P_{2}(\bt)$ realize $\OvAss{2}{10}$ as the normal fan of a polytope.
		
		\item Standard positions for $P_{4}(\bt)$ realize $\OvAss{2}{n}$ as a complete fan for all $n\le 13$.
	\end{enumerate}
\end{lemma}

\begin{proof}
	For part (1), by Lemma~\ref{lemma:standard} we only need to check that $n=8$ works and $n=9$ does not. For $n=8$ Table~\ref{table:2n} 
	shows values of $(f_{ij})_{i,j}$ that prove the fan polytopal. For $n=9$ the computer said that the system is not feasible, which finishes the proof of part (1)).
	
	For part (2), near-lexicographic positions gave the feasible solution displayed in Table~\ref{table:2n}.
	
	For part (3), the computer checked the conditions for a complete fan for $n=13$. Only the last one would really be needed; this last one took about 2 days of computing in a standard laptop.
\end{proof}

\begin{table}
	\begin{tabular}{c|c}
		$i,j$ & $f_{ij}$ \\ \hline
		3,6 & 16 \\
		3,7 & 35 \\
		3,8 & 59 \\
		4,7 & 11 \\
		4,8 & 36 \\
		5,8 & 37 \\
	\end{tabular}
	\qquad\quad
	\begin{tabular}{c|c}
		$i,j$ & $f_{ij}$ \\ \hline
		3,6 & 2097152 \\
		3,7 & 2116197 \\
		3,8 & 2116816 \\
		3,9 & 2116875 \\
		3,10 & 2116899 \\
	\end{tabular}\quad
	\begin{tabular}{c|c}
		$i,j$ & $f_{ij}$ \\ \hline
		4,7 & 1094032 \\
		4,8 & 1133231 \\
		4,9 & 1136343 \\
		4,10 & 1137410 \\
		5,8 & 5949943 \\ 
	\end{tabular}\quad
\begin{tabular}{c|c}
	$i,j$ & $f_{ij}$ \\ \hline
	5,9 & 6427739 \\
	5,10 & 6575138 \\
	6,9 & 166833470 \\
	6,10 & 237316440 \\
	7,10 & 120253293274
\end{tabular}
	\medskip
	\label{res9}
	\caption{Height vectors $(f_{ij})_{i,j}$ realizing $\OvAss{2}{8}$ (left) as a polytopal fan with rays in $P_2(1,2,3,4,5,1,2,3,4,5)$ (standard positions), and $\OvAss{2}{10}$ (right) with rays in $P_2(\bt)$, with $t_i=2^{(i-1)^2}$.}
	\label{table:2n}
\end{table}

For $k=3$ and $n=11$, the restrictions given by Theorem \ref{thm:star} are complicated to solve, but we have found a position that realizes the fan (but not the polytope): $(0,1,31,32,42,67,100)$ at both sides.


\begin{thebibliography}{99}

	\bibitem{bcl} Nantel Bergeron, Cesar Ceballos, Jean-Philippe Labb\'e, {Fan realizations of subword complexes and multi-associahedra via Gale duality}, \textit{Discrete  Comput. Geom.} 54(1),  (2015) 195--231.
	
	\bibitem{bokpil} J\"urgen Bokowski, Vincent Pilaud, {On symmetric realizations of the simplicial complex of 3-crossing-free sets of diagonals of the octagon}, \textit{Proceedings of the 21st Annual Canadian Conference on Computational Geometry, Vancouver, British Columbia, Canada, August 17-19, 2009}.

	\bibitem{BLS} 
	Marie-Charlotte Brandenburg, Georg Loho, Rainer Sinn,
	Tropical Positivity and Determinantal Varieties,
	prerpint 2022.
	\href{https://arXiv.org/abs/2205.14972}{arXiv:2205.14972}
	
	\bibitem{cp-92}
	Vasilis~Capoyleas and Janos~Pach.
	\newblock A {T}ur\'an-type theorem on chords of a convex polygon.
	\newblock {\em J. Combin. Theory Ser. B}, 56(1):9--15, 1992.
	
	\bibitem{CLS14}
	Cesar Ceballos, Jean-Philippe Labb\'e, and Christian Stump. Subword complexes, cluster complexes, and generalized multi-associahedra. \emph{J. Algebraic Combin.}, 39(1):17--51, 2014.

	\bibitem{CreSan:moment} Luis Crespo Ruiz, Francisco Santos, \textit{Bar-and-joint rigidity on the moment curve coincides with cofactor rigidity on a conic}, 
	\emph{Combinatorial Theory}, {\bf  3}(1) (2023), \#15, 13pp.
	
	\bibitem{CreSan:Pfaffians} Luis Crespo Ruiz, Francisco Santos, \textit{Multitriangulations and tropical Pfaffians}, preprint \href{https://arxiv.org/abs/2203.04633}{arXiv:2203.04633}, 2022
	
	\bibitem{CreSan:Multi} Luis Crespo Ruiz, Francisco Santos, \textit{Realizations of multiassociahedra via rigidity}, preprint
	\href{https://arxiv.org/abs/2212.14265}{arXiv:2212.14265}, 2022.
	
	\bibitem{triangbook} Jes\'us A. De Loera, J\"org Rambau, Francisco Santos, \textit{Triangulations: Structures for Algorithms and Applications}, Springer, 2012.
	
	\bibitem{DKM}
	Andreas Dress, Jack H. Koolen and Vincent Moulton. On line arrangements in the hyperbolic plane. \emph{Eur. J. Comb.}, \textbf{23}(5) (2002), 549--557.
	
	\bibitem{Graver} 
	Jack E. Graver, Rigidity matroids, \emph{SIAM Discrete Math.} {\bf 4} (1991), 355--368.
	
	\bibitem{GSS} Jack E. Graver, Brigitte Servatius, Herman Servatius, \textit{Combinatorial Rigidity}, Graduate Stud. Math. {\bf 2}, Amer. Math. Soc., Providence, RI. ISBN 0-8218-3801-6
	
	\bibitem{HPS}
	Christophe Hohlweg, Vincent Pilaud, Salvatore Stella,
	Polytopal realizations of finite type $\g$-vector fans,
	\textit{Adv.Math.} {\bf 328} (2018), 713-749,
	\url{https://doi.org/10.1016/j.aim.2018.01.019}.

	\bibitem{Jonsson} Jakob Jonsson, Generalized triangulations and diagonal-free subsets of stack polyominoes, \textit{J. Comb. Theory
	Ser.} A 112(1), 117--142 (2005).
	
	\bibitem{Kalai} Gil Kalai, Hyperconnectivity of graphs, \textit{Graphs and Combinatorics}, \textbf{1}, 65--79, 1985.
	
	\bibitem{KNN}
	Gil Kalai, Eran Nevo, Isabella Novik,
	Bipartite rigidity,
	\emph{Trans. of the Amer. Math. Soc.}. {\bf 368}(8) (2016), 5515--5545.
	\href{http://dx.doi.org/10.1090/tran/6512}{DOI:10.1090/tran/6512}
	
	\bibitem{KTT} Franz J. Király, Louis Theran, Ryota Tomioka. The Algebraic Combinatorial Approach for Low-rank Matrix Completion. \textit{Journal of Machine Learning Research} \textbf{16} (2015), 1391--1436.
	
	\bibitem{KnuMil}
	Allen Knutson and Ezra Miller. Subword complexes in Coxeter groups. \emph{Adv. Math.}, 184(1) (2004), 161--176.
	
	\bibitem{Manneville} Thibault Manneville, \textit{Fan realizations for some 2-associahedra}, Experimental Mathematics, 27(4), 377--394 (2017)

	\bibitem{MorganScott} John Morgan and Ridgeway Scott, A nodal basis for $C^1$ piecewise polynomials of degree $n\ge 5$,
	\emph{Math. Comput.} {\bf 29} (1975), 736--740.
	
	
	\bibitem{Naka}
	Tomoki Nakamigawa. A generalization of diagonal flips in a convex polygon. \emph{Theor. Comput. Sci.} \textbf{235}(2) (2000), 271--282.
	
	\bibitem{Nguyen} 
	Viet-Hang Nguyen, On abstract rigidity matroids, \emph{SIAM Discrete Math.} {\bf 24} (2010), 363--369.

%
	\bibitem{PacStu}
	Lior Pachter, Bernd Sturmfels.
	\emph{Algebraic Statistics for Computational Biology}, Cambridge University Press, 2005. \url{https://doi.org/10.1017/CBO9780511610684}.
	
	\bibitem{PilPoc} Vincent Pilaud, Michel Pocchiola, Multitriangulations, Pseudotriangulations and Primitive Sorting Networks. \emph{Discrete Comput. Geom.} \textbf{41} (2012), 142--191.
	
	\bibitem{PilSan} Vincent Pilaud, Francisco Santos, Multitriangulations as Complexes of Star Polygons, \textit{Discrete Comput. Geom.} 41 (2009), 284--317.
	
	\bibitem{PilSan:brick} Vincent Pilaud, Francisco Santos, 
	The brick polytope of a sorting network. 
	\textit{European J. Combin}. 33:4 (2012), 632--662. 
	
	\bibitem{Rosen} Zvi Rosen, Computing algebraic matroids. Preprint, \href{https://arxiv.org/abs/1403.8148v2}{arXiv:1403.8148v2}, 2014
	
	\bibitem{RST}
	Zvi Rosen, Jessica Sidman and Louis Theran,
	Algebraic Matroids in Action, 
	\emph{The American Mathematical Monthly}, {\bf 127}:3 (2020), 199--216, 
	\url{https://doi.org/10.1080/00029890.2020.1689781}
	
	\bibitem{SpeStu} David Speyer and Bernd Sturmfels. The tropical Grassmannian. \textit{Adv. Geom.}, \textbf{4}(3) (2004), 389--411.
	
	\bibitem{SpeWil}
	David Speyer, Lauren Williams. 
	The Tropical Totally Positive Grassmannian. \emph{J. Algebr. Comb. 22} (2005), 189--210. \url{https://doi.org/10.1007/s10801-005-2513-3}
	
	\bibitem{Stump} Christian Stump. A new perspective on $k$-triangulations. \textit{J. Comb. Theory A} \textbf{118}(6), 1794--1800 (2011).
	
	\bibitem{Whiteley-split} Walter Whiteley, \textit{Vertex splitting in isostatic frameworks}, Structural Topology \textbf{16} (1990), 23--30.
	
	\bibitem{Whiteley} Walter Whiteley, \textit{Some Matroids from Discrete Applied Geometry}, 
	in \emph{Matroid Theory} (Joseph E. Bonin, James G. Oxley and Brigitte Servatius, Editors), 
	Contemporary Mathematics {\bf 197}, 1996, pp. 171--311.
	
\end{thebibliography}
\end{document}